\documentclass{amsart}
\usepackage{amsmath,amscd,amsthm,amsfonts,a4wide,amssymb}
\usepackage{diagrams}

\theoremstyle{plain}
\newtheorem{theorem}                 {Theorem}      [section]
\newtheorem{proposition}  [theorem]  {Proposition}
\newtheorem{corollary}    [theorem]  {Corollary}
\newtheorem{lemma}        [theorem]  {Lemma}

\theoremstyle{definition}
\newtheorem{example}      [theorem]  {Example}
\newtheorem{remark}       [theorem]  {Remark}
\newtheorem{definition}   [theorem]  {Definition}

\numberwithin{equation}{section}

\DeclareMathOperator{\Ima}{Im}
\DeclareMathOperator{\spa}{span}

\DeclareMathOperator{\rank}{rank}

\newcommand{\wt}{\widetilde}
\newcommand{\wh}{\widehat}
\newcommand{\cc}{\mathrm c}  
\newcommand{\pa}{\partial}

\def \nn{\mathbb N}
\def \zn{\mathbb Z}
\def \rn{\mathbb R}
\def \cn{\mathbb C}
\def \hn{\mathbb H}
\def \H{\mathcal H}

\def \HH{\underline{\mathcal H}}

\def \Gr{Gr}

\def \g{\mathfrak{g}}

\def \O#1{{\rm O}(#1)}
\def \SO#1{{\rm SO}(#1)}

\def \U#1{{\rm U}(#1)}

\def \Sp#1{{\rm Sp}(#1)}

\def \d{\mathrm{d}}
\def\ip#1#2{\langle#1,#2\rangle}

\def \CP#1{\mathbb{C}P^{#1}}
\def \RP#1{\mathbb{R}P^{#1}}
\def \gl#1{\mathfrak{gl}(#1)}

\newcommand{\zbar}{\bar{z}}
\newcommand{\CC}{\underline{\mathbb C}}
\newcommand{\ov}{\overline}
\newcommand{\ul}{\underline}
\newcommand{\alg}{{\rm alg}}

\newcommand{\ii}{\mathrm{i}} 

\begin{document}
\larger[1]

\begin{footnotesize}
\begin{flushright}
CP3-ORIGINS-2010-16
\end{flushright}
\end{footnotesize}
\bigskip

\title[Filtrations, factorizations and harmonic maps]{Filtrations, factorizations and explicit formulae for harmonic maps}

\author{Martin Svensson}
\author{John C. Wood}



\thanks{The first author was supported by the Danish Council for Independent Research and the Danish National Research Foundation. 
The second author thanks the Department of Mathematics and Computer Science of the University of Southern Denmark, Odense, for support and hospitality during the preparation of this work.}

\address
{Department of Mathematics \& Computer Science, University of
Southern Denmark, and CP3-Origins, Centre of Excellence for Particle Physics Phenomenology, Campusvej 55, DK-5230 Odense M, Denmark}
\email{svensson@imada.sdu.dk}

\address
{Department of Pure Mathematics, University of Leeds, Leeds LS2 9JT, Great Britain}
\email{j.c.wood@leeds.ac.uk}

\begin{abstract}  
We use filtrations of the Grassmannian model to produce explicit algebraic 
formulae for all harmonic maps of finite uniton number from a 
 Riemann surface, and so all harmonic maps from the $2$-sphere, to the unitary group for a general class of factorizations by unitons.  We show how these specialize to give explicit formulae for such harmonic maps to each of the
 classical compact Lie groups and their inner symmetric spaces --- the nonlinear $\sigma$-model of particle physics.  Our methods also give an explicit Iwasawa decomposition of the algebraic loop group.
\end{abstract}

\maketitle

\section{Introduction}
 
In \cite{uhlenbeck}, K.~Uhlenbeck showed how to construct harmonic maps from a Riemann surface to the unitary group $\U n$ by starting with a constant map and successively modifying it by a process called `adding a uniton', a sort of B\"acklund transform.  She showed that \emph{all} harmonic maps from the \emph{$2$-sphere} $S^2$ could be obtained that way.  
Harmonic maps from the $2$-sphere are particularly important for three reasons:
 (i) they are equivalent to  
harmonic maps of finite energy from the plane
\cite{sacks-uhlenbeck}; (ii) they constitute the nonlinear $\sigma$-model of particle physics, see, e.g., \cite{zak}; (iii) they are precisely the minimal branched immersions of $S^2$ in the sense of
\cite{gulliver-osserman-royden}.
More generally, Uhlenbeck's method gives all harmonic maps from other Riemann surfaces which are of
\emph{finite uniton number}; these maps are also minimal branched immersions.
Various ways of making this more explicit were given by the second author and others, e.g., \cite{wood-unitary}; however, finding the unitons
involved the solution of $\bar{\pa}$-problems, which could rarely be solved explicitly.

In \cite{ferreira-simoes-wood}, M.~J.~Ferreira,
B.~A.~Sim\~oes and the second author showed how to solve this problem, producing algebraic formulae for the unitons, and thus for all harmonic maps of finite uniton number from a surface to the unitary group.  They used the factorization essentially due to G.~Segal \cite{segal} which is dual to that used by Uhlenbeck.  They then related their formulae to the Grassmannian model of Segal.
By a completely different method in which the unitons are thought of as stationary Ward solitons, B.\ Dai and C.-L.\ Terng \cite{dai-terng} obtained explicit formulae for the unitons of the \emph{Uhlenbeck} factorization.  

In the present paper, we use filtrations of the Grassmannian model to produce explicit algebraic formulae for harmonic maps of finite uniton number for a general class of factorizations, including not only those above, but also factorizations obtained by a mixture of them and the factorization by $A_z$-images studied by the second author \cite{wood-unitary}.
On the way, we establish many useful formulae relating uniton factorizations and filtrations.

Finally, we show how to apply our methods to finding harmonic maps of finite uniton number from a surface to the
\emph{special orthogonal group $\SO n$ and the real Grassmannians}; we also find \emph{harmonic maps to the space $\SO {2m} \big/ \U m$ of orthogonal complex structures}.  Here we use a factorization by alternate Uhlenbeck and Segal steps; our formulae for such mixed factorizations then give explicit formulae for all such harmonic maps, see  Theorems \ref{th:fact-SOn} 
and \ref{th:fact-r-odd}, and Corollary \ref{cor:fact-real-Grass}.

The same methods apply to find all harmonic maps of finite uniton number from a surface to the
\emph{symplectic group $\Sp m$ and  quaternionic Grassmannians}; here the factorization is that studied by Q.~He and Y.~Shen  \cite{he-shen2}, Y.~Dong \cite{dong} and R.~Pacheco \cite{pacheco-sympl}. 

We can also find \emph{harmonic maps to the space
$\Sp m \big/ \U m$ of `quaternionic' complex structures
on $\cn^{2m}$ (equivalently, Lagrangian subspaces of\/ 
$\cn^{2m}$)}.  In this way, we obtain explicit formulae for all harmonic maps of finite uniton number from surfaces to the classical compact Lie groups and their inner symmetric spaces. 
Note that our methods could be extended to find pluriharmonic maps from K\"ahler manifolds to these spaces using ideas of \cite{ohnita-valli}.

The paper is arranged as follows.  In \S\ref{sec:alg}, we give formulae relating factorizations and filtrations which are purely algebraic; in particular, we study the two extreme filtrations of Segal and Uhlenbeck.

Then in \S \ref{sec:extended}, we discuss extended solutions of harmonic maps, their factorizations and filtrations, and study how operators in the Grassmannian model correspond to operators on the corresponding subbundles. A special case is that of extended solutions which are invariant under the `additional $S^1$-action' of C.-L. Terng discussed
in \cite[\S 7]{uhlenbeck};
as in \cite{burstall-guest}, this extends to an action of $\cn \setminus \{0\}$ which can be used to deform any harmonic map to an $S^1$-invariant one.

Our explicit formulae for harmonic maps are given in \S \ref{sec:expl}; we show how these give explicit formulae for the algebraic Iwasawa decomposition (Theorem \ref{th:Iwasawa}). In \S \ref{sec:Grass}, we see how our methods apply to give harmonic maps to complex Grassmannians.

Finally, in \S \ref{sec:real}, we apply our methods to give harmonic maps to the groups $\SO n$, $\Sp m$ and their inner symmetric spaces: the real and quaternionic Grassmannians and the spaces
$\SO {2m} \big/ \U m$ and $\Sp m \big/ \U m$.   We thus see how
to obtain explicit formulae for
all harmonic maps of finite uniton number from a surface, in particular, all harmonic maps from the $2$-sphere, into the classical Lie groups and their inner symmetric spaces.

\section{Some basic algebraic formulae}\label{sec:alg}

\subsection{The Grassmannian model of $\Omega \U n$}

For a Lie group $G$, we recall that the group of (free) loops on $G$ is given by 
$$
\Lambda G=\{\gamma:S^1\to G\ :\ \gamma\text{ is smooth}\},
$$ 
and the group of (based) loops on $G$ is given by 
$$
\Omega G=\{\gamma\in\Lambda G\ :\ \gamma(1)=e\}.
$$
where $e$ denotes the identity element of $G$.
We shall mainly consider the case when $G$ is the unitary group $\U n$, where $n$ is a fixed positive integer, or one if its subgroups: the orthogonal group or symplectic group.

We denote by $\H = \H^{(n)}$ the Hilbert space $L^2(S^1,\cn^n)$. By expanding into Fourier series, we have
$$
\H = \text{ linear closure of }
	\{\lambda^i e_j\ :\ i\in\zn,\ j=1,\dots,n\}
$$
where $\{e_1,\dots,e_n\}$ is the standard basis for $\cn^n$; in fact, $\{\lambda^i e_j\ :\ i\in\zn,\ j=1,\dots,n\}$ is a complete orthonormal system for $\H$. 

The natural action of $\U n$ on $\cn^n$ induces an action of $\Omega\U n$ on $\H$ which is isometric with respect to the $L^2$ inner product.  We consider the closed subspace 
$$
\H_+= \text{ linear closure of } \{\lambda^i e_j\ :\ i \in \nn,\ j=1,\dots,n\}
$$ 
where $\nn = \{0,1,2,\ldots\}$.
The action of $\Omega\U n$ on $\H$ induces an action of $\Omega\U n$ on subspaces of $\H$; denote by $\Gr = \Gr^{(n)}$ the orbit of $\H_+$ under that action. For a precise description of the elements of $\Gr$ we refer to \cite{pressley-segal}; here we just note that any $W\in\Gr$ is closed under multiplication by $\lambda$,\ i.e.,\ $\lambda W \subset W$, and we have a bijective map 
\begin{equation} \label{Grass-model}
\Omega\U n\ni\Phi \mapsto W = \Phi\H_+\in\Gr;
\end{equation}
we call $W$ the \emph{Grassmannian model of\/
$\Phi$}, we shall frequently use this identification.  The map \eqref{Grass-model} restricts to a bijection from the
\emph{algebraic loop group} $\Omega_{\alg}\U n$ consisting of
those $\gamma \in \Omega\U n$ given by finite Laurent series:
$\gamma = \sum_{i=s}^r \lambda^k T_k$ where $r \geq s$ are integers and the $T_k$ are $n \times n$ complex matrices, to the
set of $\lambda$-closed subspaces $W$ of $\H$
satisfying 
$$
\lambda^r\H_+ \subset W\subset\lambda^{s}\H_+
$$
for some integers $r \geq s$.
Note that such a subspace can also be thought of as a subspace of the quotient vector space
$\lambda^{s} \H_+/\lambda^r\H_+ $; this quotient space with the inner product induced from $\H$ may be naturally identified with the finite-dimensional vector space $\cn^{(r-s)n} $ equipped with its standard
Hermitian inner product.

{}From now on, let $r \in \nn$, and let $\Omega_r\U n$ denote the set  of polynomials $\Phi = \sum_{k=0}^r \lambda^k T_k$ of degree at most $r$.  Then \eqref{Grass-model} further restricts to a bijection
from $\Omega_r\U n$ to the subset $\Gr_r$ of those $W \in \Gr$ satisfying
$$
\lambda^r\H_+ \subset W \subset \H_+\,.
$$

For any $i \in \zn$, let $P_i:\H \to \cn^n$ denote the $i$'th
natural projection given by
$L = \sum \lambda^i L_i \mapsto L_i$.
  For any subspace $\alpha$ of $\cn^n$, we denote by $\pi_{\alpha} $ and $\pi_{\alpha}^{\perp}$ orthogonal projection onto $\alpha$ and its orthogonal complement $\alpha^{\perp}$, respectively.
The fundamental idea behind relating uniton factorizations and
filtrations is the following construction due to Segal \cite{segal}, though the terminology is ours.

Let $\Phi$, $\wt{\Phi} \in \Omega\U n$ and set
$W = \Phi\H_+$\,, $\wt{W} = \wt{\Phi}\H_+  \in \Gr$.
We say that $\wt{W}$ (or $\wt{\Phi}$) is obtained from
$W$ (or $\Phi$) by a \emph{$\lambda$-step} if
\begin{equation} \label{lambda-step}
\lambda \wt{W} \subset W \subset \wt{W},
\quad \text{equivalently,} \quad
W \subset \wt{W} \subset \lambda^{-1}W.
\end{equation}

\begin{lemma} \label{le:step}  Let $W = \Phi\H_+$ and
$\wt{W} = \wt{\Phi}\H_+$ where $\Phi, \wt{\Phi} \in \Omega\U n$. 
Then $\wt{W}$ is obtained by a $\lambda$-step from $W$
if and only if
$$
\Phi = \wt{\Phi}(\pi_{\alpha} + \lambda\pi_{\alpha}^{\perp})
\text{ equivalently }
\wt{\Phi} = \Phi(\pi_{\alpha} + \lambda^{-1}\pi_{\alpha}^{\perp})\,, \quad
\text{for some subspace $\alpha$}.
$$
Further,
\begin{equation} \label{alpha-from-W}
\alpha = P_0\wt{\Phi}^{-1}W \quad \text{and} \quad
\alpha^{\perp} = P_0\Phi^{-1}\lambda\wt{W}\,;
\end{equation}
conversely,
\begin{equation}\label{W-from-alpha}
W = \wt{\Phi}(\alpha) + \lambda\wt{W}
= \Phi(\alpha) + \lambda\wt{W}
\quad \text{and} \quad
\wt{W} = \wt{\Phi}(\alpha^{\perp}) + W
= \lambda^{-1} \Phi(\alpha^{\perp}) + W.
\end{equation}
\end{lemma}

Here, $P_0\Phi^{-1}\lambda\wt{W}$ means
$P_0 \bigl(\Phi^{-1}(\lambda\wt{W})\bigr)$.

\begin{proof}
Suppose that $\wt W$ is obtained from $W$ by a $\lambda$-step, so that $\lambda\wt\Phi\H_+\subset \Phi\H_+\subset\wt\Phi\H_+$\,. Then $\lambda\H_+\subset\wt\Phi^{-1}\Phi\H_+\subset\H_+$\,, which implies that 
$\wt\Phi^{-1}\Phi\H_+=\alpha+\lambda\H_+=(\pi_{\alpha}+\lambda\pi_{\alpha}^{\perp})\H_+$
for some subspace $\alpha\subset\cn^n$; hence $\Phi=\wt\Phi(\pi_{\alpha}+\lambda\pi_{\alpha}^{\perp})$. The converse is immediate, as are \eqref{alpha-from-W} and \eqref{W-from-alpha}.
\end{proof}

Thus a $\lambda$-step $W \mapsto \wt{W}$ is equivalent to a choice of subspace $\alpha$ of $\cn^n$.
Note that we do not exclude the extreme cases: (i) $\alpha = \cn^n$, then $\wt{\Phi} = \Phi$ and $\wt{W} = W$; (ii) $\alpha =$ the zero subspace, then $\wt{\Phi}= \lambda^{-1}\Phi$ and $\wt{W} = \lambda^{-1}W$.

For the rest of this section, let $W = \Phi\H_+\in\Gr_r$ where $\Phi \in \Omega_r\U n$. 
Note that, if $r=0$, then $\Phi = I$, the identity matrix, so 
that $W = \H_+$\,.  

\begin{definition} By a \emph{$\lambda$-filtration} $(W_i)$ of $W$ we mean a nested sequence
\begin{equation*}
W= W_r\subset W_{r-1}\subset\dots\subset W_0=\H_+
\end{equation*}
of $\lambda$-closed subspaces of $\H_+$ with each $W_{i-1}$ is obtained from $W_i$ by a $\lambda$-step, i.e.,
\begin{equation} \label{filt-conds}
\lambda W_{i-1}\subset W_i \subset W_{i-1} \qquad (i=1,\ldots,r).
\end{equation}
\end{definition}

By a simple induction starting with $W_0=\H_+$ we see that each $W_i \in \Gr_i$\,.   We now identify the subspaces and loops associated to a $\lambda$-filtration.

\begin{proposition}\label{pr:factorizations}
Let $(W_i)$ be a $\lambda$-filtration of\/ $W$. Define a sequence $\Phi_i \in \Omega\U n$ inductively by
$\Phi_0=I$ and $\Phi_i=\Phi_{i-1}(\pi_{\alpha_i}+\lambda\pi_{\alpha_i}^{\perp})$
\ $(i=1,\ldots, r)$
where
\begin{equation} \label{alpha_i}
\alpha_i=P_0\Phi_{i-1}^{-1}W_i\,.
\end{equation}
Then 
$
\Phi_i \H_+=W_i\,.
$ 
\end{proposition}

\begin{proof} We use induction on $i$.  For $i=0$, it is trivial.  For $i=1$, \eqref{filt-conds} implies that
$W_1=V+\lambda
\H_+$ for some subspace $V\subset\cn^n$ and \eqref{alpha_i} gives
$\alpha_1=P_0W_1=V$.
Hence 
$$
W_1= \pi_{\alpha_1}\H_+ + \lambda\H_+ = (\pi_{\alpha_1}+\lambda\pi_{\alpha_1}^{\perp})\H_+ = \Phi_1 \H_+\,,
$$ 
as desired.

Now suppose that $\Phi_{i-1}\H_+=W_{i-1}$ for some $i > 1$.
Then, from \eqref{filt-conds} and the induction hypothesis,
$
\lambda\H_+\subset\Phi_{i-1}^{-1}W_{i}\subset \H_+\,,
$
so that, by \eqref{alpha_i},
$$
\Phi_{i-1}^{-1}W_{i}=\alpha_{i}+\lambda
\H_+=(\pi_{i}+\lambda\pi_{i}^{\perp})\H_+\,.
$$
Hence
$\Phi_{i}\H_+ = W_{i}$\,,
completing the induction step.
\end{proof}

The proposition implies that
\begin{equation} \label{Phi_i}
\Phi_0 = I, \quad \text{and} \quad \Phi_i= 
(\pi_{\alpha_1}+\lambda\pi_{\alpha_1}^{\perp})\cdots(\pi_{\alpha_i}+\lambda\pi_{\alpha_i}^{\perp}) \quad (i=1,2,\ldots, r)\,;
\end{equation}
thus, the $\Phi_i$ are polynomials in $\lambda$ of the form 
\begin{equation}\label{Phi-FS}
\Phi_i=T_0^i+\lambda T_1^i+\dots+\lambda^i T_i^i \qquad (\lambda \in S^1)
\end{equation}
where the $T_j^i$ are $n \times n$ complex matrices. 

The proposition shows that the choice of a $\lambda$-filtration $(W_i)$ of $W$ is equivalent to the choice of a sequence
$(\alpha_i)$ of subspaces of $\cn^n$; this is equivalent, in turn, to a \emph{factorization of\/ $\Phi$}:
\begin{equation} \label{Phi-fact}
\Phi = (\pi_{\alpha_1}+\lambda\pi_{\alpha_1}^{\perp})\cdots(\pi_{\alpha_r}+\lambda\pi_{\alpha_r}^{\perp}).
\end{equation}
Indeed, given $(W_i)$, define the sequence
$(\alpha_i)$ by
\eqref{alpha_i}; conversely, given an arbitrary sequence
$(\alpha_i)$ of subspaces,
define the sequence $(\Phi_i)$ by \eqref{Phi_i} and then set $W_i = \Phi_i\H_+$\,.
{}From Lemma \ref{le:step} we obtain the following formulae.

\begin{corollary}\label{co:factorizations}
Let $(W_i)$ be a $\lambda$-filtration.  Then
for $i = 1,\ldots, r$, we have 
\begin{itemize}
\item[(i)] $\alpha_i^{\perp}=P_0\Phi_i^{-1}(\lambda W_{i-1})\,;$
\item[(ii)] $W_i=\Phi_{i-1}(\alpha_i)\oplus\lambda W_{i-1}=\Phi_i(\alpha_i)\oplus\lambda W_{i-1}\,;$
\item[(iii)] $W_i=\Phi_i(\alpha_{i+1}^{\perp})\oplus W_{i+1}=\lambda^{-1}\Phi_{i+1}(\alpha_{i+1}^{\perp})\oplus W_{i+1}\,;$
\end{itemize}
Furthermore,  all the direct sums are orthogonal direct sums with respect to  the $L^2$ inner product on $\H_+$\,. 
\qed \end{corollary}

{}From \eqref{Phi-FS} we obtain
\begin{equation*} 
\Phi_i^{\;-1} = {\Phi_i}^* = S_0^i + \lambda^{-1} S_1^i +\dots+\lambda^{-i}S_i^i \qquad (\lambda \in S^1),
\end{equation*}
where each $S^i_s$ is the \emph{adjoint} $(T_s^i)^*$ of $T_s^i$\,.
On the other hand, from \eqref{Phi_i} we obtain
\begin{equation*}
\Phi_i^{\;-1} = (\pi_{\alpha_i}+ \lambda^{-1} \pi_{\alpha_i}^{\perp}) \cdots
	(\pi_{\alpha_1}+ \lambda^{-1} \pi_{\alpha_1}^{\perp})\,.
\end{equation*}
Comparing these, we see that $S^{i}_{s}$ is the sum of all $i$-fold products
of the form $\Pi_i \cdots \Pi_1$  where exactly $s$ of the $\Pi_j$ are $\pi_{\alpha_j}^{\perp}$ and the other $i-s$ are $\pi_{\alpha_j}$\,.
  
\begin{corollary} \label{co:explicit}
We have the following explicit formulae for each subspace $\alpha_i\!:$
\begin{equation} \label{alpha-and-perp}
\text{\rm (a)} \quad \alpha_i
=\Bigl( \sum_{s=0}^{i-1}S_s^{i-1}P_s \Bigr) W_i\,,
\qquad
\text{\rm (b)} \quad \alpha_i^{\perp}
= \Bigl(\sum_{s=1}^i S_s^i P_{s-1}\Bigr)W_{i-1}\,.
\end{equation}
\end{corollary}

\begin{proof}
The formulae are obtained by expanding \eqref{alpha_i} and using Corollary \ref{co:factorizations}(i).  
\end{proof} 

\subsection{Two extreme filtrations}  \label{subsec:extreme}
There are two natural $\lambda$-steps, which we shall call the \emph{Segal} and \emph{Uhlenbeck steps}, given on a subspace $W \in \Gr_i$ by 
\begin{equation*} 
W^S_{i-1}=W+\lambda^{i-1}\H_+\,, \qquad 
	W^U_{i-1}=(\lambda^{-1}W)\cap\H_+
	= (\lambda^{-1}W)\cap\H_+ +\lambda^{i-1}\H_+\,,
\end{equation*}
respectively; note that the Segal step depends on $i$.
Since
$(\lambda^{-1}W\cap \H_+)+\lambda^{i-2}\H_+
	=\bigl(\lambda^{-1}(W+\lambda^{i-1}\H_+)\bigr)\cap \H_+\,,$ the Segal and Uhlenbeck steps \emph{commute}.  

Starting with a subspace $W\in\Gr_r$
and iterating these steps gives
$\lambda$-filtrations of $W$ which appear in the work of Segal \cite{segal} and Uhlenbeck \cite{uhlenbeck}:
\begin{eqnarray}
W_i^S&=&W+\lambda^i\H_+ \quad (i=0,\dots,r)
\quad(\textit{the Segal filtration});
	\label{Seg-filt}
\\
W_i^U&=&(\lambda^{i-r}W)\cap \H_+ \quad (i=0,\dots,r)
\quad(\textit{the Uhlenbeck filtration}).
	\label{Uhl-filt} 
\end{eqnarray}

We call the corresponding subspaces $\alpha_i$ and factorization
\eqref{Phi-fact} the \emph{Segal} (resp.\ \emph{Uhlenbeck}) \emph{subspaces} and \emph{factorization}.
The following proposition shows that these are the two extremes of the possible filtrations of $W$.
  
\begin{proposition}\label{pr:extreme}
Let $W \in \Gr_r$\,.
For any $\lambda$-filtration $(W_i)$ of\/ $W$, we have
$$
W_i^S\subset W_i\subset W_i^U\qquad(i=0,\dots,r).
$$ 
\end{proposition}

\begin{proof} Since $\lambda^i\H_+\subset W_i$ and
$W\subset W_i$, we see that
$$
W_i^S=W+\lambda^i\H_+\subset W_i \qquad (i=0,\ldots, r).
$$
To show that $W_i\subset
W_i^U$, we use reversed induction: since $W_r=W_r^U =W$, it is true for
$i=r$. Assume that it is true for some $i$. Then we see that
$$
W_{i-1}\subset \lambda^{-1}W_i\subset
\lambda^{-1}(\lambda^{i-r}W)=\lambda^{i-1-r}W.
$$
Since $W_{i-1}\subset \H_+$ it follows that 
$
W_{i-1}\subset(\lambda^{i-1-r}W)\cap \H_+ =W_{i-1}^U\,,
$
and the induction step is complete. 
\end{proof}

\begin{remark}\label{rem:duality}
For any $i \geq 1$ and $W \in \Gr_i$\,, set
$W^I = \lambda^{i-1}\ov{W}^{\perp}$ where $\ov{W}$ denotes the complex conjugate of $W$.
 If $W=\Phi\H_+$\,, then clearly $W^I=\lambda^i\ov\Phi\H_+$; it follows that $W \mapsto W^I$ is an \emph{involution} on $\Gr_i$.  Furthermore, (i) if $\wt{W}$ is obtained from $W$ by a
$\lambda$-step with subspace $\alpha$, then $\wt{W}^I$ is obtained from $W^I$ by a $\lambda$-step with subspace $\ov{\alpha}^{\perp}$; (ii) if the step $W \mapsto \wt{W}$ is Segal (resp. Uhlenbeck) then the step $W^I \mapsto \wt{W}^I$ is Uhlenbeck (resp. Segal).

This involution induces an involution on $\lambda$-filtrations: given a $\lambda$-filtration $(W_i)$,
setting $W_i^I =  \lambda^{i-1}\ov{W_i}^{\perp}$  defines another $\lambda$-filtration $(W_i^I)$.
See also Example \ref{ex:factorization}.
\end{remark}

We now see what choices of subspace correspond to Segal and Uhlenbeck steps.  

\begin{proposition} \label{prop:Seg-Uhl}
For $i \geq 1$, let $\Phi \in \Omega_i\U n$.  Write
\begin{equation} \label{Phi-T-S}
\Phi = T_0 + T_1 \lambda + \cdots + T_i \lambda^i
\quad \text{so that} \quad 
\Phi^{-1} = S_0 + S_1 \lambda^{-1} + \cdots + S_i \lambda^{-i}
\end{equation}
where $S_j$ is the adjoint of\/ $T_j$ \ $(j=1,\ldots, i)$.
Let $\alpha$ be a subspace of\/ $\cn^n$. 
Write $W = \Phi\H_+$\,,
$\wt{\Phi} = \Phi(\pi_{\alpha} + \lambda^{-1} \pi_{\alpha}^{\perp})$ and $\wt{W} = \wt{\Phi}\H_+$\,.  
Then
\begin{itemize}
\item[(i)] $\wt{W}=W+\lambda^{i-1}\H_+$ \emph{(Segal step)}
if and only if\/ $\alpha=\ker(T_i)\,;$

\item[(ii)]$\wt{W}=(\lambda^{-1}W)\cap \H_+$
\emph{(Uhlenbeck step)} if and only if\/
$\alpha=\Ima(S_0)$.
\end{itemize}
(Note that we do not insist that $T_i$ or $S_0$ be non-zero.) 
\end{proposition}

\begin{proof}
(i) $\wt{W}=W+\lambda^{i-1}\H_+$ if and only if
$\Phi^{-1}(\lambda\wt{W}) = \lambda\H_++\lambda^i\Phi^{-1}\H_+$\,.
Since $\wt{W} \subset W$ and $\lambda\H_+\subset\Phi^{-1}(\lambda\wt{W})$,
this is equivalent to 
$P_0\Phi^{-1}(\lambda\wt{W}) = P_0\lambda^i\Phi^{-1}\H_+$\,. 
By Corollary \ref{co:factorizations}(i), this holds if and only if 
$$
\alpha^{\perp}=P_0\lambda^i\Phi^{-1}\H_+=\Ima(S_i),
	\quad \text{equivalently,}\quad
		\alpha=\ker(T_i).
$$

(ii) This follows from (i) by applying the involution of Remark \ref{rem:duality}.
\end{proof}

The following result shows how the particular choices of extreme filtrations correspond to `covering' properties of the corresponding subspaces.
 
\begin{proposition}\label{prop:Seg-Uhl-cov}
Let $W \in \Gr_r$.
Let $(W_i)$ be a $\lambda$-filtration of\/ $W$
and let 
$\alpha_i$ be the corresponding subspaces given by
Proposition \ref{pr:factorizations}.

{\rm (i)} Suppose that, for some $i=1,\dots,r-1$, we have
$W_{i-1}=W_i+\lambda^{i-1}\H_+$\,.  Then $W_i=W_{i+1}+\lambda^i\H_+$
  if and only if
\begin{equation} \label{Seg-cov}
\pi_{\alpha_i}(\alpha_{i+1})=\alpha_i\,.
\end{equation}
In particular, $(W_i)$ is the Segal filtration of\/ $W$ if and only if
\eqref{Seg-cov} holds for all $i = 1,\ldots, r-1$.

{\rm (ii)} Suppose that, for some $i=1,\dots,r-1$, we have
$W_{i-1}=(\lambda^{-1}W_i)\cap \H_+$\,.  Then $W_i=(\lambda^{-1}W_{i+1})\cap \H_+$ if and only if 
\begin{equation} \label{Uhl-cov}
\pi_{\alpha_{i+1}}(\alpha_i)=\alpha_{i+1}\,. 
\end{equation}
In particular, $(W_i)$ is the Uhlenbeck filtration of\/ $W$ if and only if
\eqref{Uhl-cov} holds for all $i = 1,\ldots, r-1$.
\end{proposition}

\begin{proof} (i) By Corollary \ref{co:factorizations}(ii), we have 
\begin{eqnarray}
W_{i+1}+\lambda^i\H_+
	&=&\Phi_i(\alpha_{i+1})+\lambda W_i+\lambda^i\H_+
=\Phi_i(\alpha_{i+1})+\lambda(W_i+\lambda^{i-1}\H_+) \notag
\\
	&=& \Phi_i(\alpha_{i+1})+\lambda W_{i-1}
=\Phi_{i-1}\bigl((\pi_{\alpha_i}+\lambda\pi_{\alpha_i}^{\perp})(\alpha_{i+1})+\lambda \H_+ \bigr) \notag
\\
	&=& \Phi_{i-1}\bigl(\pi_{\alpha_i}(\alpha_{i+1})+\lambda \H_+ \bigr).
		\label{W_i+1}
\end{eqnarray}

Now, if $W_i= W_{i+1} + \lambda^i \H_+$\,, then applying $\Phi_{i-1}^{-1}$ to the above gives
$$
\Phi_{i-1}^{-1}W_i=\pi_{\alpha_i}(\alpha_{i+1})+\lambda \H_+\,.
$$ 
By Proposition \ref{pr:factorizations}, $P_0\Phi_{i-1}^{-1}W_i=\alpha_i$\,, so that the last equation implies \eqref{Seg-cov}.

Conversely, if \eqref{Seg-cov} holds, then the right-hand side of \eqref{W_i+1} equals
$$
\Phi_{i-1}(\alpha_i + \lambda \H_+) = \Phi_i\H_+ = W_i\,,
$$
which establishes (i). The proof of (ii) is similar.  
\end{proof}

\subsection{$S^1$-invariant polynomials} \label{subsec:invt}

Recall (e.g.\ \cite[\S 7]{uhlenbeck}) that there is an $S^1$-action on $\Omega\U n$ given by 
$(\mu^*\Phi)_{\lambda}=\Phi_{\mu\lambda}\Phi_{\mu}^{-1}$
	\ $\bigl(\mu\in S^1, \ \Phi\in\Omega\U n \bigr).$
We now identify all $S^1$-invariant polynomials, i.e., polynomials
$\Phi \in \Omega\U n$ which satisfy
$\Phi_{\lambda}\Phi_{\mu}=\Phi_{\lambda\mu}$
	\ $(\lambda, \mu\in S^1)$.

The next result follows easily from \cite[\S 10]{uhlenbeck}.

\begin{proposition}\label{pr:alg-iso}
$\Phi\in\Omega_r\U n$ so that $W = \Phi\H_+\in\Gr_r$.
Denote by $\beta_1,\ldots,\beta_r$ and $\gamma_1,\ldots,\gamma_r$ the subspaces of\/ $\cn^n$ corresponding to the Segal and the Uhlenbeck filtrations of\/ $W$, respectively, so that
\begin{equation} \label{Seg-Uhl-fact} 
\Phi = (\pi_{\beta_1}+\lambda\pi_{\beta_1}^{\perp})\cdots(\pi_{\beta_r}+\lambda\pi_{\beta_r}^{\perp}) 
	 =(\pi_{\gamma_1}+\lambda\pi_{\gamma_1}^{\perp})\cdots(\pi_{\gamma_r}+\lambda\pi_{\gamma_r}^{\perp}). 
\end{equation}
Then the following are equivalent: 
\begin{itemize}
\item[(i)] $\beta_i\subset\beta_{i+1}$ \ $(i=1,\dots,r-1);$
\item[(ii)] $\Phi$ is $S^1$-invariant$;$
\item[(iii)] $W=\sum_{i=0}^{r-1}\lambda^i\beta_{i+1}+\lambda^r\H_+;$
\item[(iv)] $W=\sum_{i=0}^{r-1}\lambda^iP_iW+\lambda^r\H_+;$
\item[(v)] $\gamma_{i+1}\subset\gamma_i$ \ $(i=1,\dots,r-1).$ 
\end{itemize}
Furthermore, if any of the above hold, then $\gamma_i=\beta_{r-i+1}$ for all $i=1,\dots,r$. 
\qed
\end{proposition}

\begin{corollary} \label{co:S1-preserve}
Let $\Phi \in \Omega_r\U n$ be $S^1$-invariant.
Let $\wt{W} = \wt{\Phi}\H_+ \in \Omega_{r-1}\U n$ be obtained from $W = \Phi\H_+$ be a Segal or Uhlenbeck step.  Then $\wt{\Phi}$ is also $S^1$-invariant.
\end{corollary}

\begin{proof}
It is easy to see that, if
$W = \Phi\H_+$ satisfies (iv), then
$W^S_{r-1} = W + \lambda^{r-1}\H_+$ and
$W^U_{i-1} = (\lambda^{-1}W) \cap \H_+$ continue to satisfy (iv) (with $r$ replaced by $r-1$).
\end{proof}

For $S^1$-invariant harmonic maps, see
\S \ref{subsec:s1-invariant}.

\section{Harmonic maps and extended solutions}\label{sec:extended}

\subsection{Basic facts} \label{subsec:basic}
We review some well-known facts about harmonic maps, extended solutions, and their Grassmannian models; our main references are \cite{uhlenbeck}, \cite{guest-book} and \cite{burstall-guest}. {}From now on,
$M$ will denote a Riemann surface, and $G$ will denote $\U n$ or a compact Lie subgroup
of $\U n$, equipped with the natural bi-invariant metric from $\U n$. All maps and sections are assumed smooth unless otherwise stated. For any complex vector space $V$, we denote by $\ul{V}$ the trivial bundle $M\times V$ over $M$

For any map $\varphi:M\to G$, we define a $1$-form with values in the Lie algebra $\g$ of $G$ as half the pull-back of the Maurer-Cartan form, i.e., 
$
A^{\varphi}=\frac{1}{2}\varphi^{-1}\d\varphi.
$

Let $\U n$ act on $\cn^n$ in the standard way.
Then $D^{\varphi} = \d+A^{\varphi}$
defines a unitary connection on the trivial bundle $\CC^n$.
We decompose $A^{\varphi}$ and $D^{\varphi}$ into types: to do this, for convenience we take a local complex coordinate $z$ on an open set $U$ of $M$ and write
$\d\varphi = \varphi_z \d z + \varphi_{\zbar}\d\zbar$,
$A = A^{\varphi}_z \d z +  A^{\varphi}_{\zbar} \d\zbar$,
$D^{\varphi} = D^{\varphi}_z \d z + D^{\varphi}_{\zbar} \d\zbar$,
$\pa_z = \pa/\pa z$ and $\pa_{\zbar} = \pa/\pa\zbar$; then
$$
A^{\varphi}_z=\frac{1}{2}\varphi^{-1}\varphi_z\,,\quad A^{\varphi}_{\zbar}=\frac{1}{2}\varphi^{-1}\varphi_{\zbar}\,, \quad
D^{\varphi}_z = \pa_z + A^{\varphi}_z \,,\quad
D^{\varphi}_{\zbar} = \pa_{\zbar} + A^{\varphi}_{\zbar} \,.
$$

The \emph{(Koszul--Malgrange) holomorphic structure induced by $\varphi$} is the unique holomorphic structure on $\CC^n$ with $\bar{\pa}$-operator given locally by $D^{\varphi}_{\bar z}$;  we denote the resulting holomorphic vector bundle by $(\CC^n, D^{\varphi}_{\bar z})$.   If $\varphi$ is constant, $D^{\varphi}_{\bar z} = \pa_{\zbar}$, giving $\CC^n$ the standard (product) holomorphic structure. 
Now \cite{uhlenbeck} a map $\varphi:M\to G$ is harmonic if and only if
$A^{\varphi}_z$ is a holomorphic endomorphism of the holomorphic vector bundle
$(\CC^n, D^{\varphi}_{\bar z})$.  In particular its image
 and kernel form holomorphic subbundles of
$(\CC^n, D^{\varphi}_{\bar z})$, defined away from the discrete subset of $M$ where the rank of $A^{\varphi}_z$ drops; it is clear that these subbundles are independent of the local complex coordinate $z$.
By `filling out zeros' as in \cite[Proposition 2.2]{burstall-wood}, these
image and kernel subbundles can be extended to holomorphic subbundles over the whole
of $M$, which we shall denote by
$\Ima A^{\varphi}_z$ and $\ker A^{\varphi}_z$, respectively. 

Let $\g^{\cn}$ denote the complexified Lie algebra $\g \otimes \cn$.
 
\begin{definition} A smooth map $\Phi:M\to\Omega G$ is said to be an \emph{extended solution} if, with respect to any local holomorphic
coordinate $z$ on $U \subset M$, we have 
$$
\Phi^{-1}\Phi_z=(1-\lambda^{-1})A,
$$
for some map $A:U\to\g^{\cn}$. 
\end{definition}

For any map $\Phi:M \to \Omega G$ and $\lambda \in S^1$, we define
$\Phi_{\lambda}:M \to G$  by $\Phi_{\lambda}(p) = \Phi(p)(\lambda)$ \ $(p \in M)$.
If $\Phi:M\to\Omega G$ is an extended solution,  the map
$\varphi=\Phi_{-1}:M\to G$ is harmonic
and
$\varphi^{-1}\varphi_z=2A$, so that $A = A^{\varphi}_z$.

Conversely, given a harmonic map $\varphi:M\to G$, an extended solution $\Phi:M\to\Omega G$ satisfying
$$
\Phi^{-1}\Phi_z=(1-\lambda^{-1})A^{\varphi}_z
$$
is  said to be
\emph{associated} to $\varphi$. In this case, $\varphi= g \/\Phi_{-1}$
for some $g\in G$.  Extended solutions always exist locally;
they exist globally if the domain $M$ is simply-connected, for example if $M = S^2$. Further, \emph{uniqueness} is achieved by specifying an initial value $\Phi(z_0) \in \Omega G$ for some $z_0 \in M$.
Note that any two extended solutions  $\Phi$, $\wt{\Phi}$ associated to the same map $\varphi$ differ by a loop, i.e., $\Phi = \gamma \Phi$ for some
$\gamma \in \Omega G$.

For any $N \in \nn$, let $G_*(\cn^N)$ denote the Grassmannian of subspaces of $\cn^N$; thus $G_*(\cn^N)$ is the disjoint union of the complex Grassmannians   $G_k(\cn^N)$ for $k \in \{0,1,\ldots, N\}$.
We shall frequently identify a map $W:M \to G_k(\cn^N)$ with the rank $k$ subbundle 
of $\CC^N = M \times \cn^N$ whose fibre at $p \in M$ is $W(p)$; we denote this subbundle also by $W$ (not underlining, in contrast to \cite{burstall-wood,ferreira-simoes-wood}).

For a smooth map $\Phi:M\to\Omega \U n$, set
$W=\Phi\H_+:M\to\Gr$. It is easy to see that $\Phi$ is an extended solution if and only if $W$ satisfies the two conditions:
\begin{equation} \label{W-ext-sol} 
\text{(a) } \ \pa_{\zbar} \sigma \in \Gamma(W),
	\quad \text{(b) } \ \lambda\, \pa_z\sigma \in \Gamma(W)
		\qquad \bigl(\sigma\in\Gamma(W) \bigr);
\end{equation}
here $\Gamma(\cdot)$ denotes the space of smooth sections of a vector bundle.   Condition (a) says that $W$ is a holomorphic subbundle of the trivial holomorphic bundle
$\HH = (M \times \H, \pa_{\zbar})$,
and condition (b) says that it is closed under the operator  $F:\Gamma(\HH) \to \Gamma(\HH)$ given by
\begin{equation} \label{F}
F = \lambda \pa_z\,,
	\quad \text{i.e.,}
		\quad F(\sigma)=\lambda\,\pa_z \sigma
		\quad \bigl(\sigma \in \Gamma(\HH)\bigr).
\end{equation}

Conversely, if $W:M \to \Gr$ is a map satisfying
conditions \eqref{W-ext-sol}, then $W=\Phi \HH_+$ for some extended solution $\Phi:M\to\Omega \U n$.
We shall therefore call both $W$ and $\Phi$
extended solutions; we shall also refer to $W$
as the \emph{Grassmannian model} of $\Phi$.   

An extended solution is called \emph{algebraic} if it has a finite Laurent expansion
$\Phi = \sum_{k=s}^r \lambda^k T_k$ where $r\geq s$ are integers and the $T_k:M \to \gl{n,\cn}$ are smooth maps.  An argument of
 Uhlenbeck \cite[Theorem 11.5]{uhlenbeck} shows that, if $M$ is compact and
$\varphi:M \to \U n$ admits an associated extended solution, then it has an algebraic associated 
extended solution $\Phi$.  Indeed, fix a base point $z_0 \in M$; then
the extended solution satisfying the initial condition
$\Phi_{\lambda}(z_0) = I$ \ $(\lambda \in S^1)$
is algebraic, see
\cite[Theorem 4.2]{ohnita-valli} where this is extended to pluriharmonic maps.  In particular, any harmonic map $\varphi:S^2 \to \U n$ has an
algebraic associated extended solution.

There is a one-to-one correspondence between algebraic extended
 solutions $\Phi$  and extended solutions $W$ satisfying 
$\lambda^r \HH_+\subset W\subset \lambda^{s}\HH_+$
for some integers $r \geq s$ (which depend on $W$).  Note that we can think of $W$ as a subbundle of the trivial bundle
$M \times (\lambda^{s} \HH_+/\lambda^r\HH_+)$, and this may be canonically identified with the trivial holomorphic bundle
$\CC^{(r-s)n} = (M \times \cn^{(r-s)n}, \pa_{\zbar})$.

Let $\varphi:M \to \U n$ be a harmonic map.
Then a subbundle $\alpha$ of $\CC^n$ is said to be a \emph{uniton for $\varphi$} if it is
\begin{itemize}
\item[(i)] holomorphic with respect to the Koszul--Malgrange holomorphic structure induced by $\varphi$, i.e., 
$D^{\varphi}_{\zbar}(\sigma)\in\Gamma(\alpha)\quad(\sigma\in\Gamma(\alpha));$

\item[(ii)] closed under the endomorphism $A^{\varphi}_z$, i.e., 
$A^{\varphi}_z(\sigma)\in\Gamma(\alpha)\quad(\sigma\in\Gamma(\alpha))$.
\end{itemize}

Let $\varphi:M\to\U n$ be a harmonic map. Uhlenbeck showed \cite{uhlenbeck} that if a subbundle $\alpha\subset\CC^n$ is a uniton for $\varphi$ then $\wt\varphi=\varphi(\pi_{\alpha}-\pi_{\alpha}^{\perp})$ is harmonic.

\begin{example} Any holomorphic subbundle of\/ $(\CC^n,D^{\varphi}_{\zbar})$ contained in $\ker A^{\varphi}_z$ is a uniton for $\varphi$; we call such a uniton \emph{basic}. Any holomorphic subbundle of\/ $(\CC^n,D^{\varphi}_{\zbar})$ containing $\Ima A^{\varphi}_z$ is also a uniton; following Piette and Zakrzewski, see \cite{zak}, we call such a uniton \emph{antibasic}.
\end{example}

\begin{example} \label{ex:cartan-emb}
It is well known (see \cite{burstall-guest}) that any connected compact inner symmetric space can be immersed in a Lie group $G$ as a component of 
$
\sqrt{e}=\{g\in G\ :\ g^2=e\},
$
and the immersion is totally geodesic. For example, when $G = \U n$, then $\sqrt{e}=\{g\in G\ :\ g^2=e\}$ is the disjoint union $G_*(\cn^n)$ of the complex Grassmannians
$G_k(\cn^n)$ for $k \in \{0,1,\ldots, n\}$, and we have the totally geodesic Cartan embedding 
\begin{equation} \label{cartan}
\iota:G_*(\cn^n)\hookrightarrow \U n,\quad \iota(V)=\pi_V-\pi_{V}^{\perp}.
\end{equation}
Note that $\iota(V^{\perp}) = -\iota(V)$; however, we shall normally write $\iota(V)$ simply as $V$.

Let $\varphi:M \to G_*(\cn^n)$ be a smooth map and $\alpha$
a subbundle of $\CC^n$,  Then \cite{uhlenbeck},
$\wt{\varphi} = \varphi(\pi_{\alpha}-\pi_{\alpha}^{\perp})$ has image in a Grassmannian if and only if $\pi_{\alpha}$ commutes with $\pi_{\varphi}$, and this holds if and only if $\alpha$ is the direct sum of subbundles $\beta$ and $\gamma$ of $\varphi$ and $\varphi^{\perp}$, in which case
$\wt{\varphi} = \beta^{\perp} \!\cap \varphi\, \oplus \gamma$.

An important special case is when $\beta = \varphi$ and $\gamma = \Ima A^{\varphi}_z|_{\varphi}$, in which case
$\alpha$ is a uniton and $\wt{\varphi}$ is called the
\emph{$\pa'$-Gauss transform} $G'(\varphi)$ of $\varphi$; see
\cite{burstall-wood} for another description.    
\end{example}

We say that $\varphi$ is of \emph{finite uniton number} if, for some $r \in \nn$, we can write it as
\begin{equation} \label{phi-fact}
\varphi = \varphi_0(\pi_{\alpha_1}-\pi_{\alpha_1}^{\perp})\cdots(\pi_{\alpha_r}-\pi_{\alpha_r}^{\perp})
\end{equation}
where $\varphi_0 \in \U n$ is constant, and each $\alpha_i$ is a uniton for the partial
product
\begin{equation} \label{phi_i}
\varphi_{i-1} = \varphi_0(\pi_{\alpha_1}-\pi_{\alpha_1}^{\perp})\cdots
	(\pi_{\alpha_{i-1}}-\pi_{\alpha_{i-1}}^{\perp})\,.
\end{equation}
The minimum value of $r$ for which \eqref{phi-fact} holds is called the \emph{(minimal) uniton number of\/ $\varphi$}.
Uhlenbeck showed that any harmonic map from a compact Riemann surface to $\U n$ which admits an associated extended solution, in particular, any harmonic map from $S^2$ to $\U n$, has finite (minimal) uniton number at most $n-1$.

Now suppose that $\Phi$ is any extended solution associated to $\varphi$, then Uhlenbeck showed further that $\alpha$ is a uniton for $\varphi$ if and only if  $\wt\Phi=\Phi(\pi_{\alpha}+\lambda\pi_{\alpha}^{\perp})$ is also an extended solution
(associated to $\wt\varphi=\varphi(\pi_{\alpha}-\pi_{\alpha}^{\perp})$\,).
We shall therefore also say that $\alpha$ is a \emph{uniton for\/ $\Phi$}.  

\begin{definition} \label{def:uniton-fact}
Let $\Phi:M \to \Omega\U n$ be a polynomial extended solution.
By a \emph{uniton factorization of\/ $\Phi$} we mean a
product
\begin{equation}\label{Phi-fact2}
\Phi=(\pi_{\alpha_1}+\lambda\pi_{\alpha_1}^{\perp})\cdots(\pi_{\alpha_r}+\lambda\pi_{\alpha_r}^{\perp}),
\end{equation}
where each $\alpha_i$ is a uniton for 
$
\Phi_{i-1} = (\pi_{\alpha_1}+\lambda\pi_{\alpha_1}^{\perp})\cdots
(\pi_{\alpha_{i-1}}+\lambda\pi_{\alpha_{i-1}}^{\perp}).
$
\end{definition}

Note that, if $\varphi$ is given by \eqref{phi-fact}, then \eqref{Phi-fact2} is an extended solution for it; in fact,
for each $i$, the map $\Phi_{i-1}$
is an extended solution for \eqref{phi_i}\,.

Set $W=\Phi\H_+$\,. Then
\emph{a uniton factorization of $\Phi$ is equivalent to a
$\lambda$-filtration $(W_i)$ of\/ $W$ with each $W_i$ an extended solution}, the
 equivalence is given by $W_i = \Phi_i \HH_+$\,.
{}From now on, by a \emph{$\lambda$-filtration
of an extended solution $W$}, we shall mean a
$\lambda$-filtration $(W_i)$ by subbundles of $\CC^n$ where each subbundle $W_i$ in the filtration is an extended solution. That such filtrations exist is shown by the following example.

\begin{example}
Given an extended solution $W:M \to \Gr_r$, all the subbundles
$W_i^S = W + \lambda^i\HH_+$ and
$W_i^U = \lambda^{-i} W\cap\HH_+$ in the Segal and
 Uhlenbeck filtrations \eqref{Seg-filt}, \eqref{Uhl-filt} of $W$ are extended solutions, and the corresponding subbundles $\alpha_i$ are unitons, which we call the \emph{Segal} and \emph{Uhlenbeck unitons}, respectively.  Denoting these by $\beta_i$ and $\gamma_i$, we call the factorizations in \eqref{Seg-Uhl-fact} the \emph{Segal} and \emph{Uhlenbeck factorizations}.
 
For a different type of factorization, see Example
\ref{ex:Gauss-filtr}. 
\end{example}

It follows that any polynomial extended
solution $\Phi$ has a factorization into unitons, thus \emph{a harmonic map $\varphi$ from a Riemann surface to $\U n$ is of finite uniton number if and only if it has a  polynomial associated extended solution}.  As above, this holds when $M = S^2$,
or when $M$ is compact and $\varphi$ admits some (not necessarily algebraic) associated extended solution. 

\begin{remark} \label{rem:fi-uniton}
If $\varphi$ has finite uniton number, then any associated extended solution $\Phi$ which satisfies an initial condition $\Phi(z_0) = Q$ for some $z_0 \in M$ and $Q \in \Omega_{\alg}\U n$ is algebraic.
Indeed, by hypothesis, $\varphi$ admits a polynomial associated extended solution $\wt{\Phi}$, and the associated extended solution $\Phi$ with $\Phi(z_0) = Q$ is given by $\Phi = Q\wt{\Phi}(z_0)^{-1}\wt{\Phi}$, which is manifestly algebraic.
\end{remark}

Note that all the algebraic formulae of Section \ref{sec:alg} for filtrations and their associated factorizations apply to
$\lambda$-filtrations of an extended solution, with the subbundles $\alpha_i$ now unitons.  In particular, the formulae for the Segal and Uhlenbeck unitons in Proposition \ref{prop:Seg-Uhl} give these as
$\ker(T_i^i)$ and $\Ima(S^i_0)$, respectively; the next lemma ensures that these are well-defined after filling out zeros.

\begin{lemma} \label{le:hol-TS} \cite{he-shen1}
Let $\Phi:M \to \Omega G$ be an extended solution given by
\eqref{Phi-T-S}.  Then

 \ \!{\rm(i)} $T_i^i$ is a holomorphic endomorphism from $(\CC^n,D^{\varphi}_{\zbar})$
to $(\CC^n,\pa_{\zbar});$

{\rm (ii)} $S^i_0$ is a holomorphic endomorphism from $(\CC^n,\pa_{\zbar})$ to $(\CC^n,D^{\varphi}_{\zbar})$.
\qed \end{lemma}

\begin{example}\label{ex:factorization} 
Let $\Phi:M \to \Omega_r\U n$ be an extended solution and consider the map 
$\Psi=\lambda^r\ov{\Phi}:M \to \Omega_r\U n$ obtained by applying the
involution of Remark \ref{rem:duality}.
This is easily seen to be an extended solution associated to the harmonic map $\ov{\varphi}$ where $\varphi = \Phi_{-1}$\,.
A uniton factorization \eqref{Phi-fact2} of $\Phi$ into  is equivalent to the factorization
$\Psi= (\pi_{\beta_1} + \lambda \pi_{\beta_1}^{\perp}) \cdots (\pi_{\beta_r} + \lambda \pi_{\beta_r}^{\perp})$
where $\beta_i=\ov\alpha_i^{\perp}$.  If the factorization of $\Phi$ is Segal (resp.\ Uhlenbeck) then the factorization of $\Psi$ is Uhlenbeck (resp.\ Segal). 
\end{example}

\subsection{Correspondence of operators under extended solutions}

As usual, \emph{let $\Phi:M\to\Omega\U n$ be an extended solution associated to a harmonic
map $\varphi$ and $W=\Phi \HH_+$ its Grassmannian model}. 
Note that $\Phi$ gives a linear bundle-isomorphism from
 $\HH_+$ to $W$, and this induces a linear isomorphism between the spaces of sections $\Gamma(\HH_+)$ and $\Gamma(W)$ which we continue to denote by $\Phi$.  

Consider the following three operators on $\Gamma(W)$:
(i) $\lambda$ induced by the linear map $W \to W$, \
$w \mapsto \lambda w$,
(ii) $\pa_{\zbar}$ defined by $\sigma
\mapsto \pa_{\zbar}\sigma$ \ $\bigl(\sigma \in \Gamma(W)\bigr)$, and
(iii)  $F = \lambda \pa_z$
defined by $\sigma\mapsto \lambda \pa_z\sigma$ \ $\bigl(\sigma \in \Gamma(W)\bigr)$  as in \eqref{F}.
In the next result, we see how these give rise to operators on $\Gamma(\HH_+)$.

\begin{proposition}\label{pr:operators}
Under the isomorphism $\Phi$, the operators $\lambda$, $\pa_{\zbar}$ and $F$ on $\Gamma(W)$ correspond to the
following operators on $\Gamma(\HH_+):$
$$
{\rm (i)} \
\Phi^{-1} \!\circ\! \lambda \!\circ\! \Phi = \lambda\,;, \
{\rm (ii)} \ \Phi^{-1} \!\circ\! \pa_{\zbar} \!\circ\! \Phi=D^{\varphi}_{\zbar} - \lambda A^{\varphi}_{\zbar}\,; \
{\rm (iii)} \ \Phi^{-1} \!\circ\! F \!\circ\! \Phi=\lambda D_z^{\varphi}-A_z^{\varphi}\,.
$$

In particular, the three operators induce the following operators on $\Gamma(\CC^n):$
$$
{\rm (i)} \ P_0 \!\circ\! \Phi^{-1} \!\circ\! \lambda \!\circ\! \Phi = 0\,; \
{\rm (ii)} \ P_0 \!\circ\! \Phi^{-1} \!\circ\! \pa_{\zbar} \!\circ\! \Phi=D^{\varphi}_{\zbar}\,;\
{\rm (iii)} \ P_0 \!\circ\! \Phi^{-1} \!\circ\! F \!\circ\! \Phi = -A^{\varphi}_z\,.
$$
 \end{proposition}

\begin{proof}
(i) is trivial.
(ii)  For a section $f\in\Gamma(\HH_+)$ we have
\begin{align*}
(\Phi^{-1} \circ \pa_{\zbar} \circ \Phi)(f) &= 
\Phi^{-1}\bigl((\pa_{\zbar}\Phi)(f)
	+ \Phi(\pa_{\zbar}f)\bigr)= (\Phi^{-1}\pa_{\zbar}\Phi)(f) + \pa_{\zbar}f\\
	&=(1-\lambda)A^{\varphi}_{\zbar}f + \pa_{\zbar}f=D^{\varphi}_{\zbar}f- \lambda A^{\varphi}_{\zbar}f\,.
\end{align*}

(iii) is similar.
\end{proof}

Note that (ii) (and (iii)) express the well-known fact that $\Phi$ gauges  the flat connection induced by $\Phi_{\lambda}$ to the standard connection. 

\begin{corollary} \label{co:hol-ker}
Let $(W_i)$ be a $\lambda$-filtration of\/  $W$. Then the map
$P_0\circ\Phi_i^{-1}:(W_i,\pa_{\zbar}) \to (\CC^n, D^{\varphi_i}_{\zbar})$ is holomorphic and sends
{\rm (i)} $W_i$ onto $\CC^n$ with kernel $\lambda W_i\,;$ \
{\rm (ii)} $\lambda W_{i-1}$ onto $\alpha_i^{\perp}$ with kernel $\lambda W_i\,;$ \ 
{\rm (iii)} $W_{i+1}$ onto $\alpha_{i+1}$ with kernel $\lambda W_i$.
\end{corollary}
The corollary is illustrated by the following diagram, where all the vertical maps are surjections.
\begin{diagram}[height=4ex]
\lambda W_i & \subset & \lambda W_{i-1} & \subset & W_i & \supset & W_{i+1} & \supset & \lambda W_i \\
\dTo & & \dTo & & \dTo^{P_0 \circ \Phi_i^{-1}} & & \dTo & & \dTo \\
\ul{0} & \subset & \alpha_i^{\perp} & \subset & \CC^n & \supset & \alpha_{i+1} & \supset & \ul{0}
\end{diagram}

\begin{proof} Holomorphicity follows from Proposition \ref{pr:operators}(iii); the rest follows from \eqref{alpha_i} and Corollary \ref{co:factorizations}(i).
\end{proof}

\begin{lemma}\label{le:F} Let $W=\Phi\H_+:M\to\Gr_i$ and $W_{i-1}=\Phi_{i-1}\H_+:M\to\Gr_{i-1}$ be extended solutions
and let $\varphi_{i-1}$ be a harmonic map with associated extended solution $\Phi_{i-1}$. Suppose that $\Phi=\Phi_{i-1}(\pi_{\alpha}+\lambda\pi_{\alpha}^{\perp})$ for some uniton $\alpha$ for $\varphi_{i-1}$. Then 
\begin{itemize}
\item[(i)] $FW\subset\lambda W_{i-1}$ if and only if\/ $\alpha$ is a basic uniton for $\varphi_{i-1}\,;$
\item[(ii)] $FW_{i-1}\subset W$ if and only if\/ $\alpha$ is an antibasic uniton for $\varphi_{i-1}$.
\end{itemize}
\end{lemma}

\begin{proof} Part (i) follows from the fact that $\Phi_{i-1}^{-1}F\Phi\H_+\subset \lambda\HH_+$ if and only if $A_{z}^{\varphi_{i-1}}(\alpha)=0$, and (ii) from the fact that $\Phi_{i-1}^{-1}F\Phi_{i-1}\H_+\subset\alpha+\lambda\HH_+$ if and only if $\Ima A_{z}^{\varphi_{i-1}}\subset\alpha$. 
\end{proof}

The last result is illustrated by the following two commutative diagrams:
\begin{diagram}[height=4ex]
\Gamma(W) & \rTo^F & \Gamma(\lambda W_{i-1}) & & & \Gamma(W_{i-1}) & \rTo^F  & \Gamma(W)	
	\\
 \dTo & & \dTo & & &
	\dTo & & \dTo
	\\
\Gamma(\alpha) & \rTo_{A_z^{\varphi_{i-1}}} & \Gamma(\ul{0}) & & & \Gamma(\CC^n)  & \rTo_{A_z^{\varphi_{i-1}}} & \Gamma(\alpha)
	\end{diagram}
where the vertical arrows are given by $P_0 \circ (\Phi_{i-1})^{-1}$ and are surjective.

\begin{proposition}\label{pr:F-Seg-Uhl} Let $W=\Phi\H_+:M\to\Gr_{i}$ be an extended solution. Define $W_{i-1}^S = \Phi^S_{i-1}\H_+$ and $W_{i-1}^U=\Phi^U_{i-1}\H_+$ by the Segal and Uhlenbeck steps: $W_{i-1}^S=W+\lambda^{i-1}\HH_+$ and $W_{i-1}^U=(\lambda^{-1}W)\cap\HH_+\,,$ respectively, so that  $\Phi = \Phi^S_{i-1}(\pi_{\beta}+\lambda\pi_{\beta}^{\perp})=\Phi^U_{i-1}(\pi_{\gamma}+\lambda\pi_{\gamma}^{\perp})$ for some unitons $\beta$ and $\gamma$.  Then $\beta$ is antibasic and $\gamma$ is basic.  
\end{proposition}

\begin{proof} We have 
$FW\subset W\cap\lambda \HH_+=\lambda (\lambda^{-1}W\cap \HH_+)=\lambda W_{i-1}^U$, and 
$F W_{i-1}^S \subset W+\lambda^i\HH_+=W$. The proposition now follows from Lemma \ref{le:F}. 
\end{proof}
	
\subsection{$S^1$-invariant harmonic maps and superhorizontal sequences}\label{subsec:s1-invariant} 

An important special case of the above constructions is when the unitons are \emph{nested}.  We saw in \S   \ref{subsec:invt} that, algebraically, this corresponds to maps $\Phi$ invariant under the $S^1$-action.   In fact, the sequence of Segal unitons has the following property. 

\begin{definition} \label{def:superhor}
Let $\ul{0} = \delta_0 \subset \delta_1 \subset \cdots \subset \delta_r \subset \delta_{r+1} = \CC^n$ be a nested sequence of subbundles of a trivial bundle $\CC^n = M \times \cn^n$.  Say that the sequence is \emph{superhorizontal} if
\begin{enumerate}
\item[(i)] each subbundle is holomorphic with respect to the standard complex structure, i.e., $\pa_{\zbar} \sigma \in \Gamma(\delta_i)$ for all $i$ and $\sigma \in \Gamma(\delta_i)$;

\item[(ii)] the operator $\pa_z$ maps smooth sections of $\delta_i$ into sections of $\delta_{i+1}$, i.e., $\pa_z \sigma \in \Gamma(\delta_{i+1})$ for all $i$ and $\sigma \in \Gamma(\delta_i)$.
\end{enumerate}
\end{definition} 
A superhorizontal sequence is equivalent to a superhorizontal holomorphic map from $M$ to a flag manifold of $\U n$, see
 \cite[Chapter 4]{burstall-rawnsley}.
 Write $\zeta_i = \delta_{i}^{\perp} \cap \delta_{i+1}$ \ $(i= 0,\ldots, r)$, so that the $\zeta_i$ are orthogonal and have sum $\CC^n$.
 The next result follows from Proposition \ref{pr:alg-iso} (ii) and (iii) and is essentially known.

\begin{proposition} \label{pr:S1-invt}
Let $\Phi:M \to\Omega_r \U n$ be an $S^1$-invariant
 polynomial extended solution; and write $\varphi = \Phi_{-1}$.   Let $\delta_1,\dots,\delta_r$ be the corresponding Segal unitons. Then
\begin{enumerate}
\item[(i)]  the sequence
$\ul{0} = \delta_0 \subset \delta_1 \subset \cdots \subset \delta_r \subset \delta_{r+1} = \CC^n$
is superhorizontal\/$;$

\item[(ii)] the Uhlenbeck unitons satisfy $\gamma_i=\delta_{r+1-i}\,;$

\item[(iii)]  $\varphi$ maps into a Grassmannian, and is given by 
\end{enumerate}
\begin{equation} \label{phi-nested}
\varphi = \sum_{k=0}^{[r/2]}\zeta_{2k}=\zeta_0\oplus\zeta_2\oplus\dots.
\end{equation} \\[-5ex]
\qed
\end{proposition}

\begin{example} \label{ex:superhor}
(i) Any harmonic map $\varphi$ of uniton number $r=2$ to a complex Grassmannian is $S^1$-invariant, it suffices to observe that $P_1W = P_1(W \cap \lambda\HH_+)$\,.   In fact,
$\varphi = \delta_1 \oplus \delta_2^{\perp}$ is a \emph{mixed pair} in the sense of \cite[\S 3.4]{burstall-wood}, and
$\varphi^{\perp} = \delta_1^{\perp} \cap \delta_2$ is \emph{strongly isotropic} in the sense of \cite{erdem-wood}.  All harmonic maps from $S^2$ to  $\cn P^{n-1}$ are strongly isotropic.
In general, \eqref{phi-nested} expresses $\varphi$ or its orthogonal complement  as the sum of strongly isotropic harmonic maps $\zeta_i$\,.

(ii) Let $f:M \to \cn P^{n-1}$ be a full holomorphic map, then setting $\delta_i = f_{(i-1)}$ gives a superhorizontal sequence with $\rank(\delta_i) = i$.
\end{example}
 
\begin{example}  \label{ex:S1-limit}
Let $\Phi:M\to\Omega_r\U n$ be an extended solution and write $W=\Phi\H_+$\,.
Set 
$  
\delta_{i+1}=P_i(W\cap\lambda^i\HH_+)=P_0(\lambda^{-i}W\cap\HH_+)\quad(i=0,\dots, r)\,,
$
 so that $\delta_i \subset \delta_{i+1}$.
Then the sequence $(\delta_i)$  is superhorizontal.
Set $W^0 = \sum_{i=0}^{r-1}\lambda^i \delta_{i+1}+\lambda^r\HH_+$ and 
define $\Phi^0:M \to \Omega_r\U n$ by $W^0 = \Phi^0 \HH_+$\,.
Then $\Phi^0$ is an
$S^1$-invariant extended solution with Segal unitons $\delta_i$\,.
 
Now there is an action of $\cn \setminus \{0\}$ on extended solutions $W = \Phi\H_+$ which we write as $\Phi \mapsto s \cdot \Phi$ or $W \mapsto s \cdot W$
\ $(s \in \cn \setminus \{0\})$
given by
$(s\cdot W)_{\lambda} = W_{\lambda s}$; when $s \in S^1$, this is the $S^1$-action of \S \ref{subsec:invt}.
As $s \to 0$, $W$ tends to $W^0$;
thus, \emph{for any extended solution $W=\Phi\H_+$\,, $s \mapsto s \cdot \Phi$ \ $(s \in (0,1])$ gives a deformation of\/ $\Phi$ through extended solutions; as $s \to 0$, this tends
to the $S^1$-invariant extended solution $\Phi^0$}; we shall call $W^0$ and $\Phi^0$ the \emph{$S^1$-invariant limit} of
$W$ and $\Phi$.
For the action of $\cn \setminus \{0\}$ on the loop group of an arbitrary compact Lie group, see
\cite[\S 2]{burstall-guest}. 
\end{example}

\subsection{Normalized extended solutions} \label{subsec:normalized}

In the sequel, for subspaces $A, B$ of an inner product space with $B \subset A$, we write $A \ominus B$ for
$A \cap B^{\perp}$. Note that $A \ominus B$ can be canonically identified with the quotient space $A/B$.
 
Let $\Phi:M \to \Omega_r\U n$ be an extended solution and set
$W=\Phi\H_+:M\to\Gr_r$.  The $L^2$ inner product on $\HH_+$ restricts to ones on $W$ and $W \ominus \lambda W$. On giving
$W/\lambda W$ the quotient inner product, the natural isomorphism $W \ominus \lambda W \to W/\lambda W$ is an isometry.

Consider the filtration 
\begin{equation*}
W\supset W\cap\lambda \HH_+\supset W\cap\lambda^2\HH_+\supset\cdots\supset W\cap\lambda^r\HH_+
	\supset W\cap\lambda^{r+1}\HH_+ \,.
\end{equation*}
On applying the natural projection $\pi:W \to W/\lambda W$, this induces a filtration:
$W\big/\lambda W=\wh Y_0\supset\wh Y_1\supset\cdots\supset\wh Y_r\supset\wh Y_{r+1}=0$
where
$
\wh Y_i= \pi(W\cap\lambda^i\HH_+) = (W\cap\lambda^i\HH_++\lambda W) \big/(\lambda W)
	\cong
	(W\cap\lambda^i \HH_+) \big/ (\lambda W\cap\lambda^i\HH_+),
$
or, equivalently, an orthogonal decomposition:
\begin{equation} \label{Ai-decomp}
W\ominus\lambda W \cong W/\lambda W = A_0\oplus A_1\oplus\cdots\oplus A_r
\end{equation}
where 
$
A_i= \wh Y_i \ominus \wh Y_{i+1} \cong (W\cap\lambda^i\HH_+) \big/
(\lambda W\cap\lambda^i\HH_++W\cap\lambda^{i+1}\HH_+).
$

Note that $A_i \cong \delta_{i+1}/\delta_i$;
indeed, the composition of natural projections
$P_i: W\cap \lambda^i\HH_+ \to P_i(W \cap \lambda^i\HH_+)
 \to P_i(W \cap \lambda^i\HH_+) \big/P_{i-1}(W \cap \lambda^{i-1}\HH_+) = \delta_{i+1}/\delta_i$
is surjective and has kernel $\lambda W\cap\lambda^i\HH_+ +W\cap\lambda^{i+1}\HH_+$\,.
In particular,
\begin{equation} \label{sum-rank}
\sum_{i=0}^r \rank A_i = n
\end{equation} 

Recall that, if $W$ is extended solution, then for any loop
$\gamma \in \Omega\U n$, then $\gamma W$ is another extended solution,  which is \emph{equivalent} to $W$ in the sense that it is  associated to the same harmonic maps.   We next discuss when this change can be used to reduce the degree of $W$.

\begin{lemma} \label{le:reduce}
Let $W:M \to \Gr_r$ be an extended solution.
\begin{enumerate}
\item[(i)] Suppose that $\delta_i$ is constant for some
$i$.  Define $\eta\in\Omega_1\U n$ by
$\eta = \pi_{\delta_i} + \lambda \pi_{\delta_i}^{\perp}$.
Then $\eta^{-1}W :M \to \Gr_{r-1}$\,.

\item[(ii)] If $A_i = 0$, equivalently, the Segal unitons $\beta_i$ of\/ $W$ satisfy
$\rank\beta_i = \rank\beta_{i+1}$,  then $\delta_i$ is constant.
\end{enumerate}
\end{lemma}

\begin{proof} 
(i) Write $\wt{W} = \eta^{-1}W$; we must show that $\lambda^{r-1}\HH_+ \subset \wt{W} \subset \HH_+$\,.
Now 
$P_{-1}\wt{W} = \pi_{\delta_i^{\perp}}P_0W = \pi_{\delta_i^{\perp}}\delta_1 = 0$ since $\delta_1 \subset \delta_i$, hence
$\wt{W} \subset \HH_+$\,.

Further, since $W \supset \lambda^{r-1}\delta_r + \lambda^r\HH_+$ and
 $\delta_i \subset \delta_{r}$, we have
$$
P_{r-1}\wt{W} \supset
P_{r-1}\bigl((\pi_{\delta_{i}} + \lambda^{-1}\pi_{\delta_i}^{\perp}) (\lambda^{r-1}\delta_r + \lambda^r\HH_+)\bigr) = \delta_i + \delta_i^{\perp} 
= \cn^n,
$$
hence $\lambda^{r-1}\HH_+ \subset \wt{W}$.

(ii) $A_i = 0$ is equivalent to
$\delta_i = \delta_{i+1}$.  This implies that $\delta_i$ is $\pa_z$- and $\pa_{\zbar}$-closed, so constant.
Lastly, note that $\rank\delta_i = \rank\beta_i$\,.
\end{proof}

Let $W = \Phi\H_+$.  We say that $W$ (and $\Phi$) are \emph{normalized} if
$A_i\neq 0$ for all $i=0,\dots,r$,
cf.\ \cite[Theorem 4.5]{burstall-guest}.
By iterating the above reduction process, we can normalize a given extended solution, as follows, with (ii) following from \eqref{sum-rank}.

\begin{proposition}\label{pr:normalized}
{\rm (i)} Given an extended solution $W:M \to \Gr_r$, there exists an integer $s$ with $0\leq s\leq r$ and  $\gamma \in\Omega_{r-s}\U n$ such that $\gamma^{-1}W:M \to \Gr_s$ is normalized with no $\delta_i$ constant.

{\rm (ii)} Any normalized solution $\wt{W}: M \to \Gr_s$ has $s\leq n-1$. 
\qed \end{proposition}
 
\begin{remark}
Let $W \in Gr_r$ $(r >0)$ be an extended solution with Segal unitons $\beta_i$\,.   Suppose that (i) $\beta_1$ is full and $\beta_r \neq \CC^n$.  Then the ranks of the $\beta_i$ are strictly increasing (equivalently,
the ranks of the Uhlenbeck unitons are strictly decreasing), otherwise some $\delta_i$ would be constant and $\beta_1 =\delta_1$ would be contained in that $\delta_i$ contradicting fullness.
Hence $W$ is normalized.   

In particular, a polynomial extended solution $\Phi:M\to\Omega_r\U n$ is said to be of \emph{type one} if $\delta_1$ is full; if the degree of $\Phi$ is exactly $r$ then $\beta_r \neq \CC^n$. 
Uhlenbeck proved \cite{uhlenbeck} that any harmonic map $\varphi:M\to\U n$ of finite uniton number has a \emph{unique} type one associated polynomial extended solution $\Phi$, and its degree equals the minimal uniton number of $\varphi$.  The conditions (i) above are met so that the corresponding $W = \Phi\H_+$ is normalized.  
\end{remark}

\section{Explicit formulae for harmonic maps and the Iwasawa decomposition} \label{sec:expl}

\subsection{Explicit formulae for harmonic maps}
\label{subsec:expl} 

As in \cite{ferreira-simoes-wood}, we describe holomorphic subbundles of the trivial bundle
$\CC^N = (M \times \cn^n, \pa_{\zbar})$ by meromorphic functions as follows.
By a \emph{meromorphic spanning set} of a holomorphic bundle
 $E$ of rank $k$ we mean a collection $\{L^j\}$ of
 meromorphic sections of $E$ which spans the fibres of $E$
 except on a discrete set.  If, further, the set $\{L^j\}$ is linearly
 independent except on a (possibly bigger) discrete set, then as in
\cite[\S 7]{dai-terng} we call it a \emph{meromorphic frame} for $E$.
The following is well known; for clarity we give a proof.

\begin{lemma}  {\rm (i)} Any holomorphic subbundle $E$ of the trivial holomorphic bundle $\CC^N$ has a meromorphic frame
$\{L_j:j=1,\ldots,k\}$ with $k = \rank(E)$.

{\rm (ii)} Given any finite collection $\{L_j\}$ of meromorphic sections of\/ $\CC^N$, there is a unique holomorphic subbundle with meromorphic spanning set $\{L_j\}$, which we denote by $\spa\{L_j\}$.
\end{lemma}

\begin{proof} (i) Recall that $E$ defines a holomorphic map $E:M \to G_k(\cn^N)$.  The Grassmannian $G_k(\cn^N)$
has standard chart
$U \to \cn^{k(N-k)} = M_{k,N-k}(\cn)$ given by mapping the
row span of the $k \times N$ matrix $\wt{A} = \bigl[I_k | A \bigr]$ to $A$; here $I_k$ denotes the $k \times k$ identity matrix.  Permuting the columns gives the standard atlas of charts with rational transition functions. Choose a chart whose domain intersects the image of $E$ ; without loss of generality, we may assume that this is the standard chart above.
Now, the composition of this chart with the map $E$ is holomorphic except on  the discrete set $D = E^{-1}\bigl(G_k(\cn^N) \setminus U \bigr)$.  Set the $L_j$ equal to the rows of $\wt{A} \circ E$.   Since the transition functions are rational, the $L_j$ give a meromorphic frame for $E$. 

(ii) Away from the discrete set where one or more  of the $L_j$ has a pole, or the rank of the span of the $L_j$ is less than its maximal value, we  obtain a holomorphic bundle; it is clear that we can  fill in holes to extend it to the whole of $M$. 
\end{proof} 

As before, let $G_*(\cn^N)$ denote the Grassmannian of all subspaces of $\cn^N$.

\begin{definition} \label{def:assoc-curve}
Let $X$ be a holomorphic map from $M$ to $G_*(\cn^N)$, equivalently, a holomorphic subbundle
of $\CC^N$.  For $i \in \nn$, the
\emph{$i$'th osculating subbundle of\/ $X$} is
the subbundle $X_{(i)}$ of $\CC^N$ spanned by the local holomorphic sections $L$  of $X$ and their derivatives $L, L^{(1)}, \ldots, L^{(i)}$ with respect to any complex coordinate on $M$ of order up to $i$.  The corresponding map
$X_{(i)}:M \to G_*(\HH_+/\lambda^r \HH_+) \cong G_*(\cn^{rn})$ is called the
the \emph{$i$'th associated curve of\/ $X$}.
\end{definition}

Let  $W:M \to \Gr_r$  be an extended solution; thus $W=\Phi\H_+$ for some extended solution $\Phi:M \to \Omega_r\U n$.
Guest \cite{guest-update} noted that all such $W$ are given by taking an arbitrary holomorphic map
$X:M \to G_*(\H_+/\lambda^r \H_+) \cong G_*(\cn^{rn})$,
equivalently, holomorphic subbundle $X$ of
$(\CC^{rn},\pa_{\zbar})$, and setting $W$ equal to the coset
\begin{equation} \label{W}
W = X + \lambda X_{(1)} + \lambda^2 X_{(2)} + \cdots + \lambda^{r-1} X_{(r-1)}
	+ \lambda^r \HH_+ 
\end{equation}
We shall say that \emph{$X$ generates $W$}.

To find our explicit formulae, choose a meromorphic spanning set $\{L^j\}$ for the subbundle $X$ of $\CC^{nr}$. Then,
since $\lambda^{r+1} \HH_+ \subset \lambda W$,
the set $\{\lambda^k (L^j)^{\! (k)}\ :\ 0 \leq k \leq r \}$ gives a \emph{meromorphic spanning set for $W \mod \lambda W$}, by which we mean that the $\lambda^k (L^j)^{\! (k)}$ are meromorphic sections of $W$ whose cosets span $W\big/\lambda W$.

{}From Corollary \ref{co:explicit}, given a $\lambda$-filtration $(W_i)$ of $W$, the unitons are given explicitly by \eqref{alpha_i}, or equivalently \eqref{alpha-and-perp}, furthermore by Corollary \ref{co:hol-ker}, given a meromorphic spanning set of each $W_i$\,, \eqref{alpha_i} and \eqref{alpha-and-perp} give meromorphic spanning sets for each $\alpha_i$\,.
If we now specify how to get a meromorphic spanning set for $W_i$ from one for $W$, this leads to explicit formulae for the unitons, and so for the extended solution $\Phi$ .  Since
$P_0\Phi_{i-1}^{-1}(\lambda W_{i-1}) = 0$, it suffices to know a spanning set for $W_i \mod \lambda W_{i-1}$.
Note that $\lambda W_{i-1} \supset \lambda^i\HH_+$\,.

We first see how this works for the Segal filtration.

\begin{example} \label{ex:Seg-expl}
Let $W:M \to \Gr_r$ be an extended solution. For the Segal filtration \eqref{Seg-filt}, formulae \eqref{alpha_i} and \eqref{alpha-and-perp} simplify to the following formulae for the Segal unitons:
\begin{equation} \label{beta-W}
\beta_i=P_0\Phi_{i-1}^{-1}W_i^S = P_0\Phi_{i-1}^{-1}(W + \lambda^i\HH_+)
= P_0\Phi_{i-1}^{-1}W = \sum_{s=0}^{i-1}S_s^{i-1}P_s W.
\end{equation}

More explicitly, let $X$ be a holomorphic bundle generating $W$ as in \eqref{W}. Choose a meromorphic spanning set $\{L^j\}$ of $X$. 
Then, from the above,
$\{\lambda^k (L^j)^{\! (k)}\ :\ 0 \leq k \leq r-1 \}$ (note the $r-1$) gives a meromorphic spanning set of\/
$W \mod (\lambda W + \lambda^r\HH_+)$\,.
  For every $i,j$, write $P_i L^j=L^j_i$ so that
$L^j = \sum_i \lambda^i L^j_i$\,.  Note that $L^j_i = 0$ for $i < 0$ and for $i \geq r$ so that  $L^j = \sum_{i=0}^{r-1} \lambda^i L^j_i$. Then
the formula \eqref{beta-W} becomes
\begin{equation} \label{beta-L}
\beta_i= \spa \Bigl\{ \sum_{s=0}^{i-1}S_s^{i-1}(L^j_{s-k})^{\! (k)} \ :\ 
 0 \leq k \leq i-1 \Bigr\}.
\end{equation}
Note that the sum can be taken from $s=k$; this is then the formula in \cite[Proposition 4.4]{ferreira-simoes-wood}; 
it  gives the first three unitons as:

\begin{align*}
\beta_1=\spa\{&L^j_0\}\,;  \qquad 
\beta_2=\spa\{\pi_{\alpha_1} L^j_0+\pi_{\alpha_1}^{\perp} L^j_1\,,
\ \pi_{\alpha_1}^{\perp} (L^j_0)^{\! (1)} \}\,; \\
\beta_3=\spa\{&\pi_{\alpha_2} \pi_{\alpha_1} L^j_0+(\pi_{\alpha_2} \pi_{\alpha_1}^{\perp}+\pi_{\alpha_2}^{\perp}\pi_{\alpha_1}) L^j_1+\pi_{\alpha_2}^{\perp}\pi_{\alpha_1}^{\perp} L^j_2\,, \\
&(\pi_{\alpha_1}^{\perp}+\pi_{\alpha_2}^{\perp}) (L^j_0)^{\! (1)}+\pi_{\alpha_2}^{\perp}\pi_{\alpha_1}^{\perp} (L^j_1)^{\! (1)}\,,\pi_{\alpha_2}^{\perp}\pi_{\alpha_1}^{\perp} (L^j_0)^{\! (2)}\}.
\end{align*}

As shown in \cite[Lemma 4.2]{ferreira-simoes-wood},  the linear transformation $E:\HH_+ \to \HH_+$\,, given by
$H= \sum \lambda^i H_i \mapsto L = \sum \lambda^i L_i$
where
\begin{equation} \label{H-to-L}
L_i = \sum_{\ell=0}^i\binom{i}{\ell}H_{\ell} \qquad (i \in \nn)
\end{equation}
converts \eqref{beta-W} to the formula in
\cite[Theorem 1.1]{ferreira-simoes-wood}:
\begin{equation} \label{beta-H}
\beta_i= \spa \Bigl\{ \sum_{s=k}^{i-1}C_s^{i-1}(H^j_{s-k})^{\! (k)} \ :\  0 \leq k \leq i-1 \Bigr\}
\end{equation}
where $C^i_s = \sum_{1 \leq i_1 < \cdots < i_s \leq i}
	\pi_{\beta_{i_s}}^{\perp} \cdots \pi_{\beta_{i_1}}^{\perp}$.
\end{example}

These formulae give all harmonic maps of finite uniton number
from any Riemann surface $M$ to $\U n$ as follows.  For any integer $r \in \{0,1,\ldots, n-1\}$, choose
an \emph{arbitrary} $r \times n$ matrix
$(L_i^j)_{0 \leq i \leq r-1,\ 1 \leq j \leq n}$ or
$(H_i^j)_{0 \leq i \leq r-1,\ 1 \leq j \leq n}$
of meromorphic sections of $\CC^n$ and an arbitrary element $\varphi_0 \in \U n$, compute the $\beta_i$ for $i=1,2,\ldots, r$ from \eqref{beta-L} or \eqref{beta-H};
 then compute $\varphi$ from  \eqref{phi-fact}; an associated extended solution $\Phi$ is given by \eqref{Phi-fact2}.  This gives all harmonic maps and associated extended solutions of finite uniton number at most $r$ explicitly in terms of arbitrary meromorphic functions on $M$ using only algebraic operations.

More generally, we can use a mixture of Segal and Uhlenbeck steps to get new formulae, as follows.
Let $L$ be a meromorphic section of $\HH_+$\,. By the \emph{order $o(L)$ of $L$} we mean the least $i$ such that
$P_iL \neq 0$; equivalently, $L = \lambda^{o(L)} \wh{L}$ for some $\wh{L} = \sum_{i \geq 0} \lambda^i\wh{L}_{i}$ with $\wh{L}_{0}$ non-zero.

\begin{proposition} \label{pr:expl-S-U}
Let $W:M \to \Gr_r$ be an extended solution.
Suppose that $X$ generates $W$ as in \eqref{W}.
Let $\{L^j\}$ be a meromorphic spanning set for $X$; denote the order of\/ $L^j$ by $o(j)$.
Let $W_i$ be obtained from $W$ by $u$ Uhlenbeck steps and $r-i-u$ Segal steps, in any order.  Then
\begin{multline} \label{alpha-L-S-U}
\alpha_i 
= \spa\Bigl\{\sum_{s=0}^{i-1}S_s^{i-1}
(L^j_{s-k+u})^{\! (k)}  \ : \ 0 \leq o(j)+k-u \leq i-1 \Bigr\} \\[-2ex]
+ \ \spa\Bigl\{\sum_{s=0}^{i-1}S_s^{i-1}
	(\wh{L}^j_s)^{\! (k)} \ : \ o(j)+k-u <0 \Bigr\}.
\end{multline}
\end{proposition}

\begin{proof}
We have
$$
\alpha_i=P_0\Phi_{i-1}^{-1}W_i
	= \sum_{s=0}^{i-1}S_s^{i-1}P_s(W_i) \text{ where }
	W_i = \lambda^{-u} W \cap \HH^+  + \lambda^i \HH_+\,.
$$
Now, a meromorphic spanning set for $W \mod \lambda^r\HH_+$ 
is
$\{\lambda^{\ell} (L^j)^{(k)}\ :\ k \leq \ell \leq r-1-o(j)\}$,
 hence a meromorphic spanning set for
$W \cap \lambda^u \HH_+ \mod \lambda(W \cap \lambda^u \HH_+) + \lambda^r\HH_+$  is
$$
\{\lambda^k (L^j)^{\! (k)}\ :\ u \leq o(j) + k \leq r-1 \} \cup \{\lambda^{u-o(j)}(L^j)^{\! (k)}\ :\ o(j) + k < u \}.
$$ 
It follows that a meromorphic spanning set for $W_i = \lambda^{-u} W \cap \HH_+ \mod \lambda W_i + \lambda^i H_+$ is
$$
\{\lambda^{k-u} (L^j)^{\! (k)}\ :\ 0 \leq o(j) + k-u \leq i-1 \} \cup \{\lambda^{-o(j)}(L^j)^{\! (k)}\ :\ o(j) + k -u < 0\}.
$$
On noting that
$\lambda^{-o(j)}(L^j)^{\! (k)} = (\wh{L}^j)^{\! (k)}$, the proposition follows.
\end{proof}

\begin{example} \label{ex:alternating}
(i) If we do alternate Uhlenbeck and Segal steps, we obtain
the \emph{alternating filtration} with associated \emph{alternating factorization}.  This will be important in the real case (Section \ref{sec:real}.  The unitons are given
by \eqref{alpha-L-S-U} with $u = [(r-i+1)/2]$, if we start with an Uhlenbeck step, or $u = [(r-i)/2]$, if we start with a Segal step.

(ii) If we do no Uhlenbeck steps, \eqref{alpha-L-S-U}  reduces to \eqref{beta-W}.
\end{example}

We now consider the other extreme case where we do only Uhlenbeck steps.

\begin{example} \label{ex:Uhl-expl}
Let $W:M \to \Gr_r$ be an extended solution. If we take the Uhlenbeck filtration, then $u = r-i$, so that \eqref{alpha-L-S-U} reduces to the following formula for the Uhlenbeck unitons:
\begin{equation}\label{alpha-Uhl-L} 
\gamma_i = \spa\{\sum_{s=0}^{i-1}S_s^{i-1}(\wh{L}^j_s)^{\! (k)}\ :\ o(j)+k \leq r-i\}.
\end{equation}
Now set $\wh{L} = E(\wh{H})$ where $E:\HH_+ \to \HH_+$ is defined by \eqref{H-to-L}.  Then, by the same calculations as in
\cite[Lemma 4.2]{ferreira-simoes-wood}, the last formula can 
be written
\begin{equation} \label{alpha-Uhl}
\gamma_i = \spa\{\sum_{s=0}^{i-1}C_s^{i-1}(\wh{H}^j_s)^{\! (k)}\ :\
	o(j)+k \leq r-i\},
\end{equation}
which is equivalent to the formulae in \cite{dai-terng}.   As a specific example,
suppose that $r=3$ and that $X$ is spanned by
$L^1 = L^1_0+ \lambda L^1_1+ \lambda^2 L^1_2$ 
and $L^2 = \lambda^2 \wh{L}^2_0$.
Then the formula \eqref{alpha-Uhl} gives
\begin{align*} \label{Dai-T-9.4}
\gamma_1 &= \spa\{  \wh{H}^2_0\,, (H^1_0)^{\! (k)}\ :\ k \leq 2 \}\,, \qquad
\gamma_2 = \spa\{ (H^1_0)^{\! (k)}
	+ \pi_{\alpha_1}^{\perp}(H^1_1)^{(k)}\ :\ k \leq 1 \}\,, \\
\gamma_3 &= \spa\{ H^1_0 + (\pi_{\alpha_1}^{\perp} + \pi_{\alpha_2}^{\perp}) H^1_1
   + \pi_{\alpha_2}^{\perp}\pi_{\alpha_1}^{\perp} H^1_2\}         \,.
\end{align*}
These are the formulae of \cite[Example 9.4]{dai-terng}.
\end{example}

There are many other methods of factorization for which we can give explicit formulae, we mention just one.

\begin{example} \label{ex:Gauss-filtr}
Let $W\subset\HH_+$ be an extended solution.
For $i = 0,1,2, \ldots$, let $W_{(i)}$ denote the
$i$'th osculating subbundle of 
$W$  (Definition \ref{def:assoc-curve}).
{}From \eqref{W-ext-sol}, it follows that
$W \mapsto \wt{W} = W_{(1)}$ is a $\lambda$-step
(see \eqref{lambda-step}), which we shall call a \emph{Gauss step},
and each $W_{(i)}$ is an extended solution.
Suppose that $P_0 W$ is full; this is the case if $\Phi$ is the type one extended solution associated to a harmonic map of finite uniton number.
Then there is some $r \leq n$ such that
$(P_0 W)_{(r)} = \CC^n$ but $(P_0 W)_{(r-1)} \neq \CC^n$; it follows that
$W_{(r)} = \HH_+$\,.
 
For $i = 0,1,\ldots, r$, set $W_i = W_{(r-i)}$.  Then
$(W_i)$ is a $\lambda$-filtration by extended solutions.  In particular, $\Phi$ has finite uniton number at most $r$; however, the uniton number may actually be less than $r$, see below.

We identify the corresponding unitons $\alpha_i$.  {}From Corollary \ref{co:factorizations}(i) we have that 
$
\alpha_i^{\perp} = P_0\Phi_i^{-1}\lambda W_{i-1}
	= P_0\Phi_i^{-1}\lambda (W_i)_{(1)}\,.
$
Now $\lambda(W_i)_{(1)} = \lambda W_i + F W_i$, so, on using
Proposition \ref{pr:operators}, we obtain
$$
\alpha_i^{\perp}
	= P_0\Phi^{-1}_iF W_i
	= A_z^{\varphi_i} P_0\Phi_i^{-1}W_i
	= A_z^{\varphi_i}(\CC^n) = \Ima A_z^{\varphi_i}.
$$

The resulting factorization \eqref{phi-fact} is the
\emph{factorization by $A_z$-images} considered by the second author \cite{wood-unitary};
for maps into Grassmannians, it is the analogue of the
\emph{factorization by Gauss transforms} in \cite{wood-Grass}. In fact, given a harmonic map
$\varphi:M \to G_*(\cn^n)$, define its
\emph{$i$'th $\pa'$-Gauss
transform} $G^{(i)} (\varphi)$ iteratively by
$G^{(0)}(\varphi) = \varphi$  and
$G^{(i)} (\varphi)=  G'\bigr(G^{(i-1)}(\varphi)\bigr)$
so that 
$G^{(1)}(\varphi)$ is the $\pa'$-Gauss transform $G'(\varphi)$
(see Example \ref{ex:cartan-emb}).
Then, if $\Phi$ is the type one extended solution associated to a non-constant harmonic map $\varphi:S^2 \to \cn P^n$,  then the $r$ above may be anything between $1$ and $n$, and
$\varphi_i$ is the $(r-i)$th $\pa'$-Gauss transform of $\varphi$, whereas the uniton number of $\varphi$ is one if it is holomorphic or antiholomorphic, or two if it is not. 

We can give explicit formulae for the unitons in the factorization by $A_z$-images of a polynomial extended solution
$W$ in terms of a meromorphic spanning set $(L^j)$ of a subbundle $X$ which generates $W$.  Indeed, $W_i$ has a meromorphic spanning set 
$\{\lambda^k (L^j)^{\! (k+\ell)}\ :\ k, \ell \in \nn, \ 
0 \leq o(j) + k \leq i-1, \ 0 \leq \ell \leq r-i \} 
\mod \lambda W_i + \lambda^i\HH_+$\,,
so that Corollary \ref{co:explicit} gives
\begin{equation} \label{Gauss-unitons}
\alpha_i =
\spa\bigl\{\sum_{s=k}^{i-1} S_s^{i-1} (L^j_{s-k})^{\! (k+\ell)}\ :\ k, \ell \in \nn,
	\ 0 \leq o(j) + k \leq i-1, \ 0 \leq \ell \leq r-i  \bigr\} \,.
\end{equation}
\end{example} 

\subsection{Complex extended solutions and an explicit Iwasawa factorization}

For a compact Lie subgroup $G$ of $\U n$ with complexification $G^{\cc}$, let $\Lambda G^{\cc} = \{\gamma:S^1 \to G^{\cc} \, :  \,\gamma \text{  is smooth}\}$. 
With notation as in \cite{burstall-guest}, let $\Lambda^+G^{\cc}$
(resp.\ $\Lambda^*G^{\cc}$) denote the subgroup of $\Lambda G^{\cc}$ consisting of maps $S^1 \to G$ which extend holomorphically to the region $\{\lambda \in \cn\ :\ |\lambda| < 1\}$ \ (resp. \ $\{\lambda \in \cn\ :\ 0 < |\lambda| < 1\}$).
Then, by a  \emph{complex extended solution}
$\Psi:U \to \Lambda G^{\cc}$, we mean a holomorphic map
$U \to \Lambda^*G^{\cc}$ on an open subset of $M$  which satisfies
$
\lambda \Psi^{-1}\Psi_z \in \Lambda^+\g^{\cc}
$ 
where $\Lambda^+\g^{\cc}$ is the Lie algebra of $\Lambda^+G^{\cc}$.  

Let $D$ be a discrete subset of $M$ and let $\Psi: M \setminus D \to \Lambda G^{\cc}$ be a holomorphic map  which extends to a meromorphic map $\Psi:M \to \Lambda \gl{n,\cn}$.  Set
$W = \Psi\HH_+$ on $M \setminus D$; by filling out zeros, $W$ extends to a holomorphic subbundle of
$\HH = M \times \H$.  Then \emph{$W$ is an extended solution on $M$ if and only if\/ $\Psi$ is a complex extended solution on $M \setminus D$}.  Explicitly, denote the columns of $\Psi$ by $\{F_1,\ldots, F_n\}$; these give a meromorphic basis for $W \mod \lambda W$.  Conversely, given such a basis, let  
$D$ be the discrete set of points where one of the $F_i$ has a pole or the span of the $F_i$ has dimension less than $n$, and set $\Psi$ equal to the matrix with columns $F_i$\,.

Denote by $\Lambda_{\alg} G^{\cc}$ and $\Lambda_{\alg}^+G^{\cc}$ the subsets of those loops $\lambda \mapsto \gamma(\lambda)$ in $\Lambda G^{\cc}$ and
 $\Lambda^+G^{\cc}$ given by finite Laurent series and polynomials in $\lambda$, respectively, and recall the \emph{Iwasawa decomposition} \cite{pressley-segal}:
 \begin{equation*}
\Lambda_{\alg} G^{\cc} = \Omega_{\alg} G\cdot\Lambda_{\alg}^+G^{\cc}\,,
\end{equation*}
where $\Omega_{\alg}G\cap\Lambda_{\alg}^+G^{\cc}=\{e\}$;
this says that any $\gamma \in \Lambda_{\alg}G^{\cc}$ can be written uniquely as $\gamma = \gamma_u \gamma_+$ where $\gamma_u \in \Omega_{\alg}G$ and $\gamma_+ \in \Lambda_{\alg}^+G^{\cc}$.  A purely algebraic formulation of the following was given in \cite{ferreira-simoes-wood}; here we give a version for our meromorphic complex extended solutions.

\begin{theorem} \label{th:Iwasawa}
 Let $\Psi:M \setminus D \to \Lambda_{\alg} G^{\cc}$ be a complex extended solution which extends to a meromorphic map $\Psi:M \to \Lambda \gl{n,\cn}$.  
Set $\Phi = \Psi_u$ on $M \setminus D$.  Then $\Phi$ extends to an extended solution on $M$ and is given explicitly by algebraic formulae in the entries of\/ $\Psi$ as follows:
the columns of\/ $\Psi$ give a meromorphic basis for
$W \mod \lambda W;$ use \eqref{beta-L}, \eqref{alpha-L-S-U}, \eqref{alpha-Uhl-L} or \eqref{alpha-Uhl-L}
to obtain unitons $\alpha_i$, then $\Phi$ is given by the product \eqref{Phi-fact}. 

 If\/ $\Psi$ is polynomial in $\lambda$, then this gives $\Phi$ polynomial of the same or lesser degree.
\end{theorem}

\begin{proof}  We need only note that, as above, $W = \Phi\H_+$ extends to a holomorphic subbundle over $M$.
\end{proof}

\section{Maps into complex Grassmannians}\label{sec:Grass}

\subsection{Harmonic maps into complex Grassmannians}
\label{subsec:cx-Grass}

Recall the Cartan embedding \eqref{cartan} of the Grassmannian
$G_*(\cn^n)$ in $\U n$.

Consider  the involution $\nu$ on $\Omega\U n$ given by
$\nu(\eta)(\lambda)=\eta(-\lambda)\eta(-1)^{-1}$ and set 
$$
\Omega\U n^{\nu}=\{\eta\in\Omega\U n\ :\ \nu(\eta)=\eta\}.
$$
For any $i \in \nn$, the map $\nu$ restricts to an involution on $\Omega_i\U n$ with fixed point set
$\Omega_i\U n^{\nu}
= \Omega_i\U n \cap \Omega\U n^{\nu}$.
An extended solution $\Phi:M\to\Omega\U n$ lies in $\Omega\U n^{\nu}$ if and only if it it is invariant under
$\nu$, i.e.,
\begin{equation} \label{Phi-sym}
\Phi_{\lambda}\Phi_{-1}=\Phi_{-\lambda} \qquad (\lambda \in S^1)\,.
\end{equation}
Clearly, if an extended solution $\Phi$ satisfies this condition, the harmonic map $\varphi = \Phi_{-1}$ satisfies $\varphi^2 = I$, and so takes values in $G_*(\cn^n)$.  Conversely, given a harmonic map $\varphi:M\to G_*(\cn^n)$ of finite uniton number, there is a polynomial extended solution  with $\wt{\Phi}_{-1} = \varphi$. 
Indeed, by \cite[Lemma 15.1]{uhlenbeck}, the type one extended solution $\Phi$ associated to $\varphi$ satisfies $\Phi_{\lambda}=Q\Phi_{-\lambda}\Phi_{-1}^{-1}Q$ where $Q=\pi_V-\pi_{V}^{\perp}$ for some $V \in G_*(\cn^n)$, and
$\Phi_{-1}=Q\varphi$.  Setting
$Q_{\lambda}  =\pi_V+ \lambda\pi_{V}^{\perp}$, we see that
$\wt\Phi_{\lambda} = Q_{\lambda}\Phi_{\lambda}$ is a polynomial extended solution which has $\wt\Phi_{-1} = \varphi$
and  satisfies \eqref{Phi-sym}.

Let $\nu:\HH_+ \to \HH_+$ be the involution induced by $\lambda \mapsto -\lambda$; thus if $T = \sum_i T_i \lambda^i$ then
$\nu(T) = \sum_i (-1)^i T_i \lambda^i$.
This induces the involution $\nu:\Gr \to \Gr$ given by $W_{\lambda} \mapsto W_{-\lambda}$. 
 Under the identification of
$\Omega\U n$ with $\Gr$, this corresponds to the involution $\nu$ on $\Omega\U n$ defined above.
Denote by $\Gr^{\nu}$ the fixed point set of $\nu:\Gr \to \Gr$; for any $i \in \nn$, the involution $\nu$ restricts to $\Gr_i$
and has fixed point set $\Gr^{\nu}_i = \Gr^{\nu} \cap \Gr_i$\,.

Most of the following is in \cite[Theorem 15.3]{uhlenbeck}.

\begin{lemma} \label{le:Grass-preserve}
Let $\Phi:M \to \Omega\U n^{\nu}$ be an extended solution and set $W =\Phi\H_+: M \to \Gr^{\nu}$.
Set $\wt{\Phi} = \Phi(\pi_{\alpha}+\lambda^{-1}\pi_{\alpha}^{\perp})$
for some subbundle $\alpha$ of\/ $\CC^n$, and write $\wt{W} = \wt{\Phi}\H_+$\,.
Then 
\begin{enumerate}
\item[(i)] $\wt{W} \in \Gr^{\nu}$ if and only if\/ $\pi_{\alpha}$ commutes with $\Phi_{-1}$, equivalently with $\wt{\Phi}_{-1}\,;$

\item[(ii)] if\/ $\wt{W}$ is obtained from $W$ by
a Segal or Uhlenbeck step (\S \ref{subsec:extreme}), or a Gauss step (Example \ref{ex:Gauss-filtr}), then condition (i) holds;

\item[(iii)] if\/ $W$ is $S^1$-invariant, so is $\wt{W}$.
\\[-6.5ex]
\end{enumerate}
\qed \end{lemma}

Thus any of the factorizations in \S \ref{subsec:expl} give sequences of extended harmonic maps
$\Phi_i:M \to \Omega\U n^{\nu}$.

Conversely, starting with data $\{L^j\}$ where each polynomial $L^j$ has only even or odd powers of $\lambda$, we obtain all harmonic maps into the complex Grassmannian $G_*(\cn^n)$ explicitly.

Let $\varphi:M\to G_*(\cn^n)\subset\U n$ be a harmonic map of finite uniton number and let $\Phi:M\to\Omega\U n^{\nu}$ be
 a polynomial extended solution with $\Phi_{-1} = \varphi$.  Write 
$W=\Phi\H_+\in\Gr^{\nu}$.  Note that $\nu$ restricts to an isometric involution on  $W \ominus \lambda W$.
Let $W_{\pm}$ be the $(\pm 1)$-eigenspaces of $\nu$ on $W$.  Then the $(\pm 1)$-eigenspaces of $\nu$ on $W \ominus \lambda W$ are $\pi(W_+) = W_+ \ominus (\lambda W)_+$ and $\pi(W_-) = W_- \ominus (\lambda W)_-$\,, and these are orthogonal complements of each other.

Recall the filtration $(\wh Y_i)$ of $W \ominus \lambda W$ and resulting orthogonal decomposition \eqref{Ai-decomp}.

\begin{lemma}
\begin{equation} \label{even-odd-decomp}
W\ominus\lambda W = (\sum_iA_{2i})\oplus(\sum_iA_{2i-1})
\end{equation}
is the orthogonal decomposition of\/ $W\ominus\lambda W$ into the two eigenspaces of\/ $\nu$.
\end{lemma} 

\begin{proof} Let $x \in A_i$.  Then $x = \pi(y)$ for some $y \in W \cap \lambda^i \HH_+$.  
Write $y = y_+ + y_-$ where $y_{\pm} \in W_{\pm}$\,.  Then $x = x_+ + x_-$ where $x_{\pm} = \pi(y_{\pm}) \in \pi(W_{\pm})$.

If $i$ is even, $y_- \in W \cap \lambda^{i+1}\HH_+$ so that $x_- \in \wh{Y}_{i+1}$.  Since $x_+$ is orthogonal to $x_-$, we have
$0 = \ip{x}{x_-} = \ip{x_+ + x_-}{x_-} = \ip{x_-}{x_-}$ so that $x_- = 0$ and $x = x_+ = \pi(y_+) \in \pi(W_+)$.
Similarly if $i$ is odd, $x_+ = 0$ and $x = x_- \in \pi(W_-)$.
\end{proof}

Recall that $\Phi^{-1}$ restricts to an isomorphism from $W \ominus\lambda W$ to $\CC^n$;  by
\eqref{Phi-sym}, we have a commutative diagram:\\[-4ex]
\begin{diagram}[height=4ex]
W \ominus \lambda W & \rTo^{\nu} & W \ominus \lambda W \\
\dTo^{\Phi^{-1}} & & \dTo^{\Phi^{-1}} \\
\CC^n & \rTo_{\varphi = \Phi_{-1}} & \CC^n
\end{diagram}
so that the involution $\nu$ on $W \ominus\lambda W$ corresponds to the involution
$\Phi_{-1}=\varphi = \pi_{\varphi} - \pi_{\varphi}^{\perp}$ on $\CC^n$.
Under the isomorphism $\Phi:\CC^n\to W\ominus\lambda W$, the decomposition \eqref{even-odd-decomp}
corresponds to the orthogonal decomposition: $\CC^n = \varphi \oplus \varphi^{\perp}$ into the
eigenspaces of the involution $\varphi$.

We now show how the reduction procedure of Lemma
\ref{le:reduce} can be improved for Grassmannian solutions.

\begin{proposition} \label{pr:reduce-Gr}
Let $W:M \to \Gr^{\nu}_r$ be an extended solution.
Suppose that $\delta_i$ is constant for some
$i \neq 1, r$.  Define $\eta\in\Omega_2\U n^{\nu}$ by
$\eta = \pi_{\delta_i} + \lambda^2 \pi_{\delta_i}^{\perp}$;
note that $\eta_{-1} = I$.
Then $\eta^{-1}W :M \to \Gr^{\nu}_{r-2}$\,.
\end{proposition}

\begin{proof} 
We have
$P_{-2}\wt{W} = \pi_{\delta_i^{\perp}}P_0W = \pi_{\delta_i^{\perp}}\delta_1 = 0$ since $\delta_1 \subset \delta_i$.
Similarly,
$P_{-1}\wt{W} = \pi_{\delta_i}^{\perp}P_1 W$.
Since $\nu(W) = W$, we have
$P_1 W = P_1(W \cap \lambda\HH_+) = \delta_2$\,,
so that $P_{-1}\wt{W} = \pi_{\delta_i}^{\perp}\delta_2 = 0$.
Combining these shows that $\wt{W} \subset \HH_+$\,.

Further since $W \supset \lambda^{r-2}\delta_{r-1}  + \lambda^r\HH_+$ and
 $\delta_i \subset \delta_{r-1}$, we have
$$
P_{r-2}\wt{W} \supset
P_{r-2}\bigl\{(\pi_{\delta_{i}} + \lambda^{-1}\pi_{\delta_i}^{\perp}) (\lambda^{r-2}\delta_{r-1} + \lambda^r\HH_+)\bigr\}
= \delta_i + \delta_i^{\perp} = \cn^n.
$$
Hence
$\lambda^{r-2}\HH_+ \subset \wt{W}$.
\end{proof}

\begin{corollary} \label{co:reduce-Gr}
Suppose that $W:M \to \Gr_r^{\nu}$ is an extended solution with
$A_i = 0$ for some $i \neq 0, r$, then there is $\eta \in \Omega_2 \U n^{\nu}$ with $\eta_{-1} = I$ such that
$\eta^{-1}W :M \to \Gr_{r-2}^{\nu}$\,.
\qed \end{corollary}

Note that if $A_0 = 0$ if and only if $W = \lambda \wt{W}$ for some $\wt{W}:M \to \Gr_{r-1}^{\nu}$ and $A_r = 0$ if and only if
$W \in \Gr_{r-1}^{\nu}$.

\begin{corollary} Let $W:M \to \Gr_r^{\nu}$ be an extended solution, and let $\beta_1,\dots,\beta_r$ be the Segal unitons of\/ $W$. If\/ $\rank\beta_i=\rank\beta_{i+1}$ for some $i\in\{1,\dots,r-1\}$, then there is an $\eta\in\Omega_2\U n^{\nu}$ with $\eta_{-1} = I$ such that $\eta^{-1} W:M \in \Gr_{r-2}^{\nu}$\,.
\qed \end{corollary}

A similar result is true for the Uhlenbeck unitons.

\smallskip

We may iterate the reduction process of Corollary \ref{co:reduce-Gr} to obtain the following result.

\begin{proposition}\label{pr:normalized-Gr} Given an extended solution $W:M \to \Gr_r^{\nu}$, there exists an integer $s$ with $0\leq s\leq r$ and  $\eta\in\Omega_{r-s}\U n^{\nu}$ such that
$\eta^{-1}W:M \to \Gr_s^{\nu}$ is normalized with $\delta_i$ non-constant for $i \neq 0,s$. 
\qed\end{proposition}

\begin{corollary}  Let $\varphi:M\to G_k(\cn^n)$ be a harmonic map of finite uniton number.
Then there is a polynomial extended solution $\Phi:M\to\Omega\U n^{\nu}$ of degree $r\leq2\min\{k,n-k\}$, with $\Phi_{-1}=\pm\varphi$. If $2k=n$, then $\Phi$ can be chosen to have degree $r\leq 2k-1$.
\end{corollary}

\begin{proof} Without loss of generality, we may assume that $k\leq n-k$.
Let $\wt\Phi$ be any algebraic extended solution associated to $\varphi$. Fix a point $z_0\in M$ and set $\wt\Theta=\wt\Phi(z_0)^{-1}\wt\Phi$. Then, clearly, $\wt\Theta$ is still algebraic.  Let $\varphi(z_0)=V_0$ and $Q_{\lambda}=\pi_{V_0}+\lambda\pi_{V_0}^{\perp}$. Then $\Theta=Q\wt\Theta$ satisfies $\Theta_{-1}(z_0)=\varphi(z_0)$, so by uniqueness, $\Theta_{-1}=\varphi$. Furthermore, since $Q=\Theta(z_0)\in\Omega\U n^{\nu}$, it follows again by uniqueness that $\Theta:M\to\Omega\U n^{\nu}$. Finally, let $\Phi=\lambda^j\Theta$, where $j$ is chosen to make $\Phi$ is polynomial; denote its degree by $r$. By the above proposition we may assume that $W=\Phi\H_+$ is normalized; this implies that $r\leq 2k$ and, when $2k=n$, that $r\leq 2k-1$.  
\end{proof}

\begin{remark} (i) The corollary applies to any harmonic map from a \emph{compact} Riemann surface which admits an extended solution since, as noted in \S \ref{subsec:basic}, this is of finite uniton number.
Indeed, for any associated extended solution $\wt\Phi$, the extended solution $\wt\Theta =\wt\Phi(z_0)^{-1}\wt\Phi$ satisfies
$\wt\Theta(z_0) = I$, and so is algebraic by \cite{ohnita-valli}.

(ii) This bound on the minimal uniton number was originally conjectured by Uhlenbeck \cite{uhlenbeck}, and later proved using different methods by Y.\ Dong and Y.\ Shen \cite{dong-shen}.
\end{remark}

\section{Harmonic maps into $\O n$, $\Sp n$ and their inner symmetric spaces} \label{sec:real}

\subsection{Harmonic maps into $\O n$.} \label{subsec:SOn}

As usual, let $M$ be any Riemann surface.
We consider $\O n$ as a subgroup of $\U n$ and look for uniton factorizations which give harmonic maps from $M$ into this subgroup.  Note that any based loop in $\O n$, and so any extended solution, will actually have values in $\SO n$, i.e.,
$\Omega \O n = \Omega \SO n$.

\begin{definition}\label{def:real} Let $W \in \Gr$ and let $i \in \nn$.  We say that $W$ is
 \emph{real (of degree $i$)} if $W \in \Gr_i$  and
$\ov{W}^{\perp}=\lambda^{1-i}W$. We denote the set of such $W$ by $\Gr_i^{\rn}$. 
\end{definition}

For $i \in \nn$, denote by $\Omega_i\U n^{\rn}$ the set of $\Phi\in\Omega_i\U n$ for which $\ov\Phi=\lambda^{-i}\Phi$. For $\Phi\in\Omega\U n$ set $W=\Phi\H_+$; then, it is easy to see (cf. \cite[\S 8.5]{pressley-segal}) that {\it $\Phi\in\Omega_i\U n^{\rn}$ if and only if\/ $W\in\Gr_i^{\rn}$}. If $i=2j$, then this condition is equivalent to $\Psi = \lambda^{-j}\Phi$ being real in the usual sense, i.e., $\Psi\in\Omega\O n$, equivalently,
\begin{equation}\label{Phi-SOn}
\Psi = \sum_{\ell=-j}^j \lambda^{\ell} T_{\ell} \quad
\text{with} \quad T_{-\ell} = \ov{T_{\ell}} \quad \forall \ell\,.
\end{equation}
Note that this implies that $\Psi_{-1} ( = \pm \Phi_{-1} )$ belongs to $\O n$.

Now let $W_i \in \Gr_i$ with $i \geq 2$.  Define $W_{i-2}\in\Gr_{i-2}$ by 
$$
W_{i-2}=(\lambda^{-1}W_i\cap \H_+)+\lambda^{i-2}\H_+\,.
$$ 
We see that $W_{i-2}$ is obtained from $W_i$ by first doing an Uhlenbeck step and then a Segal step; since these two operations commute (see \S \ref{subsec:extreme}), $W_{i-2}$ is equally well obtained from $W_i$ by first doing a Segal step and then a Uhlenbeck step. We are primarily interested in the real case, where it is the passage from $W_i$ to $W_{i-2}$ which interests us.

\begin{proposition}  \label{pr:real-properties}
{\rm (i)} If\/ $W_i\in\Gr_i^{\rn}$, then $W_{i-2}\in\Gr_{i-2}^{\rn}$\,.

{\rm (ii)} Write $W_i=\Phi_i\H_+\,$, $W_{i-2}=\Phi_{i-2}\H_+\,$, so that
$$
\Phi_{i-2} = \Phi_i(\pi_{\alpha} + \lambda^{-1}\pi_{\alpha}^{\perp})(\pi_{\beta} + \lambda^{-1}\pi_{\beta}^{\perp})
$$ for two subspaces $\alpha, \beta$. 
\begin{itemize}
\item[(a)] If $\alpha$ is Uhlenbeck then $\alpha$ and $\beta^{\perp}$ are isotropic$;$
\item[(b)] if $\alpha$ is Segal, then $\alpha^{\perp}$ and $\beta$ are isotropic. 
\end{itemize}
\end{proposition}

\begin{proof} (i) This follows from: 
$$
\ov{W}_{i-2}^{\perp}=(\lambda\ov{W_i}^{\perp} +
\lambda\H_+)\cap\lambda^{3-i}\H_+
= (\lambda^{2-i}W_i \cap \lambda^{3-i}\H_+) + \lambda\H_+
	=\lambda^{3-i}W_{i-2}\,.
$$

(ii)(a) Let $W^U_{i-1}$ and $W^U_{i-2}$ be obtained from $W_i$ by one and two Uhlenbeck steps, respectively, and denote by $\gamma_i$ and $\gamma_{i-1}$ the corresponding Uhlenbeck subspaces.  Then $\gamma_i=\gamma$, and since by Proposition \ref{pr:extreme},
$W_{i-2}\subset W_{i-2}^U$, we obtain $(\pi_{\beta}+\lambda^{-1}\pi_{\beta}^{\perp})\H_+\subset (\pi_{\gamma_{i-1}}+\lambda^{-1}\pi_{\gamma_{i-1}}^{\perp})\H_+$\,, which implies that $\gamma_{i-1}\subset\beta$. Thus $\gamma=\pi_{\gamma}(\gamma_{i-1})=\pi_{\gamma}(\beta)$.

Since both $\Phi_i$ and $\Phi_{i-2}$ are real, the factor $\lambda^{-1}(\pi_{\beta}+\lambda\pi_{\beta}^{\perp})(\pi_{\gamma}+\lambda\pi_{\gamma}^{\perp})$ belongs to $\Omega\O n$. Hence $\pi_{\ov\beta}\pi_{\ov\gamma}=\pi_{\beta}^{\perp}\pi_{\beta}^{\perp}$; taking images we get 
$\beta^{\perp}=\pi_{\ov\beta}(\ov\gamma) \subset \ov\beta$, and hence $\beta^{\perp}$ is isotropic. On the other hand, taking the adjoint gives $\pi_{\ov\gamma}\pi_{\ov\beta}=\pi_{\beta}^{\perp}\pi_{\beta}^{\perp}$; hence $\ov\gamma=\pi_{\beta}^{\perp}(\beta^{\perp})\subset\gamma^{\perp}$, and thus $\gamma$ is also isotropic.
Part 	(ii)(b) is proved in a similar way. 
\end{proof}

Note that, if $W_i:M\to\Gr_i$ is an extended solution, so is $W_{i-2}:M\to\Gr_{i-2}$. Starting with any $W:M\to\Gr_r$  with $r \geq 2$, on performing alternate Uhlenbeck and Segal steps, we obtain the \emph{alternating filtration}
of Example \ref{ex:alternating}.
When we start with real $W$ of \emph{even} degree, iterating this process leads to a uniton factorization
(Definition \ref{def:uniton-fact}) of a real algebraic extended solution into real quadratic factors, as follows. 

\begin{theorem} \label{th:fact-SOn}
Let $r=2s$ where $s \in \nn$. Let $\Phi:M \to \Omega_r \U n^{\rn}$ be an extended solution. Then $\Phi$ has a uniton factorization$:$
$\Phi=\eta_1\cdots\eta_r$ with $\eta_i = \pi_{\alpha_i} + \lambda \pi_{\alpha_i}^{\perp}$ where
\begin{enumerate}
\item[(i)] each quadratic factor $\eta_{2\ell-1}\eta_{2\ell}$ has values in $\Omega_2\U n^{\rn}$ \ $(\ell=1,\dots,s);$

\item[(ii)] each `even' partial product $\eta_1\cdots\eta_{2j}$ has values in $\Omega_{2j}\U n^{\rn}$ \ $(j=1,\dots,s)$.
\end{enumerate}
Further, each $\alpha_i$ is given explicitly in terms of\/
 $W = \Phi\H_+$ by \eqref{alpha-L-S-U} as in Example \ref{ex:alternating}.
\end{theorem}

The hypothesis on $\Phi$ is equivalent to saying that $W$ is real of degree $r$, and this is equivalent to saying that
$\Psi=\lambda^{-s}\Phi:M \to \Omega\O n$ is an extended solution of the form \eqref{Phi-SOn} with $j=s$.
Conclusion (i) is equivalent to saying that each quadratic factor $\lambda^{-1}\eta_{2\ell-1}\eta_{2\ell}$ lies in
$\Omega\O n$ and is of the form \eqref{Phi-SOn} with $j=1$.
Conclusion (ii) is equivalent to saying that $\lambda^{-j}\eta_1\cdots\eta_{2j}:M\to\Omega\O n$ is an extended solution of the form \eqref{Phi-SOn}.

We now show how we can apply the theorem to find harmonic maps.

\begin{lemma} Let $\varphi:M\to\O n$ be a harmonic map of (minimal) uniton number at most $s$ as a map into $\U n$. Then $\varphi$ has an  associated extended solution $\Psi:M\to\Omega\O n$ of the form  $\eqref{Phi-SOn}$ with $j=s$. 
\end{lemma}

\begin{proof}
Let $\wt{\Phi}:M \to \Omega \U n$ be a polynomial extended solution of degree at most $s$ associated to $\varphi$.  Choose $z_0 \in M$ and 
set $\Psi = \wt{\Phi}(z_0)^{-1}\wt{\Phi}$.   Then $\Psi(z_0)=I=\ov\Psi(z_0)$, hence $\ov\Psi=\Psi$ by uniqueness of extended solutions. It follows that $\Psi$ is of the form \eqref{Phi-SOn}.  
\end{proof}

\subsection{Harmonic maps into real Grassmannians.}

The Cartan embedding \eqref{cartan} restricts to an identification of the union $G_*(\rn^n) = \cup_k G_k(\rn^n)$ of real Grassmannians with the totally geodesic submanifold $\{g\in\U n : g^2=I \text{ and } \ov g=g\}$ of $\U n$. Recall from \S \ref{subsec:cx-Grass} the involutions $\nu$ on $\Omega\U n$ and on $\Gr$, and their fixed point sets $\Omega\U n^{\nu}$ and $Gr^{\nu}$.  Clearly the first of these restricts to involutions on $\Omega\O n$ and on $\Omega_i\U n^{\rn}$ for any $i$; denote the corresponding fixed point sets by
$\Omega\O n^{\nu} = \Omega\O n \cap \Omega\U n^{\nu}$ and $\Omega_i\U n^{\nu,\rn} = \Omega_i\U n^{\rn} \cap \Omega\U n^{\nu}$. For any $i$, the involution $\nu$ restricts to $\Gr_i^{\rn}$, with fixed point set $\Gr_i^{\nu,\rn}=\Gr_i^{\nu}\cap\Gr_i^{\rn}$. For $\Phi\in\Omega\U n$, set $W=\Phi\H_+$\,. Then $\Phi\in\Omega_r\U n^{\nu,\rn}$ if and only if $W\in\Gr_r^{\nu,\rn}$; in this case, if $r=2s$, then $\Phi_{-1}$ is a harmonic map of uniton number at most $r$ into a real Grassmannian $G_k(\rn^n)$.
By Lemma \ref{le:Grass-preserve}, Theorem \ref{th:fact-SOn}
specializes to the following.

\begin{corollary} \label{cor:fact-real-Grass}
Let $\Phi:M \to \Omega_{2s}\U n^{\nu,\rn}$ be an extended solution. Then $\Phi$ has a uniton factorization:
$\Phi=\eta_1\cdots\eta_r$ with
$\eta_i = \pi_{\alpha_i} + \lambda \pi_{\alpha_i}^{\perp}$,
such that
 \begin{enumerate}
 \item[(i)] each quadratic factor $\eta_{2\ell-1}\eta_{2\ell}$ has values in $\Omega_2\U n^{\rn}$ \ $(\ell=1,\dots,s);$
 \item[(ii)] each partial product $\eta_1\dots\eta_i$ has values in $\Omega_i\U n^{\nu}$ \ $(i=1,\ldots, 2s);$
 \item[(iii)] each even partial product $\eta_1\cdots\eta_{2j}$ has values in
 $\Omega_{2j}\U n^{\nu,\rn}$ \ $(j=1,\dots,s)$.
\end{enumerate} 
Further, each $\alpha_i$ is given explicitly in terms of\/
 $W = \Phi\H_+$ by \eqref{alpha-L-S-U} as in Example \ref{ex:alternating}. 
\qed \end{corollary}

To apply this to harmonic maps we need the following result.

\begin{lemma}
Let $\varphi:M \to G_k(\rn^n)$ be a harmonic map to a real Grassmannian with $n-k$ even, which is of finite uniton number as a map into $\U n$.
Then $\varphi$ has an extended solution
$\Psi:M \to \Omega\O n^{\nu}$ of the form \eqref{Phi-SOn} for some
$s \in \nn$, and with $\Psi_{-1} = \varphi$. 
\end{lemma}

Note that, if $n-k$ is odd, then we can embed $G_k(\rn^n)$ in $G_k(\rn^{n+1})$. 

\begin{proof}
Let $z_0\in M$ and write $\varphi(z_0)^{\perp}=\delta+\ov\delta$, where $\delta\subset\cn^n$ is an isotropic subspace. Set
$$
Q^{\rn}_{\lambda}=\lambda^{-1}\pi_{\delta} + \pi_{(\delta+\ov\delta)}^{\perp} + \lambda\pi_{\ov\delta}\in \Omega\O n^{\nu}\,;
$$
then $Q^{\rn}_{-1}=\varphi(z_0)$.  By Remark \ref{rem:fi-uniton}, there is an extended solution $\Psi:M\to\Omega_{\alg}\U n$ associated to $\varphi$ with initial condition $\Psi(z_0)=Q^{\rn}$. Since $\ov Q^{\rn}=Q^{\rn}$ and $Q_{\lambda}^{\rn}=Q_{-\lambda}^{\rn}(Q_{-1}^{\rn})^{-1}$, by uniqueness of extended solutions we see that $\Psi:M\to\Omega\O n^{\nu}$. Now $\Psi_{-1}= g\varphi$ for some $g \in \U n$; evaluating at $z_0$ shows that $g=I$, i.e., $\Psi_{-1}=\varphi$. 
\end{proof}
 
\subsection{Harmonic maps into the space of orthogonal complex structures} \label{subsec:ocs}

By an \emph{orthogonal complex structure on $\rn^{2m}$} we mean
an isometry $J$ of $\rn^{2m}$ with $J^2=-I$.  We can choose an 
orthonormal basis $\{e_1, \ldots, e_{2m}\}$ of $\rn^{2m}$ with $e_{2j} = Je_{2j-1}$ \ $(j=1,\ldots, m)$;
$J$ is called \emph{positive} if this basis is positively oriented.  This identifies the space of orthogonal complex structures with the Hermitian symmetric space $\O {2m}/\U m$ and the space of positive orthogonal complex structures with $\SO{2m}/\U m$.
Further, the space $\O {2m}/\U m$ can be identified with the space of maximal isotropic subspaces of $\cn^{2m}$; explicitly,  extend an orthogonal complex structure $J$ to $\cn^{2m}$ by complex-linearity, and identify it with its
$(-\ii)$-eigenspace (`$(0,1)$-space') $V$; in terms of a basis $\{e_i\}$ as above, this is $\spa\{e_{2j-1}+\ii e_{2j}: j=1,\ldots, m \}$.
This gives a totally geodesic embedding of $\O {2m}/\U m$ in $G_m(\cn^{2m})$; 
with the standard conventions this embedding is holomorphic.   Composing it with the Cartan embedding of $G_m(\cn^{2m})$ into $\U {2m}$ gives the totally geodesic embedding
$J \mapsto \pi_V - \pi_V^{\perp} = \pi_V - \pi_{\ov{V}} $ with image
$\{g\in\U{2m}: g^2=I \text{ and } \ov g=-g\}$. 

In \S \ref{subsec:SOn}, we started with $W \in \Gr$ which was real and of even degree.  Let us now consider a subspace $W$ which is real of \emph{odd} degree $r$. 

\begin{example}\label{ex:r=1real} Suppose that $r=1$, so that $W=V+\lambda\H_+$\,. Then
 $W$ is real if and only if $V^{\perp}=\ov V$, i.e., $V\subset\cn^n$ is maximal
 isotropic. In particular we must have $n=2m$ for some $m$.
Now, $W$ is an extended solution if and only if $V$ is holomorphic; it follows that
\emph{$W = \Phi\H_+$ is a real extended solution of degree one if and only if\/ $\Phi_{-1}$ is a holomorphic map into $\O{2m}/\U m$}. 
\end{example}

In general, if $W$ is real of odd degree $r$, then after $r-1$ alternate Uhlenbeck and Segal steps as described in \S \ref{subsec:SOn}, we are left with a space which is real of degree one, and hence $n$ must be even. We have proved the following result.

\begin{theorem}\label{th:fact-r-odd} Let $\Phi:M\to\Omega_r\U{n}^{\rn}$ be an extended solution for some odd $r = 2s+1$.
 Then $n=2m$ for some
$m \in \nn$ and $\Phi$ has a uniton factorization
(Definition \ref{def:uniton-fact})$:$ $\Phi=\eta_0\eta_1\cdots\eta_{r-1}$ with $\eta_i = \pi_{\alpha_{i+1}} + \lambda \pi_{\alpha_{i+1}}^{\perp}$ where 
\begin{itemize}
\item[(i)] $\alpha_1$ is maximal isotropic in $\CC^{2m}$ and defines a holomorphic map $M\to\O{2m}/\U m;$
\item[(ii)] each quadratic factor $\eta_{2\ell-1}\eta_{2\ell}$ has values in $\Omega_2\U n^{\rn}$
\ $(\ell=1,\dots,s);$
\item[(iii)] each even partial product $\eta_1\cdots\eta_{2j}$ has values in $\Omega_{2j}\U n^{\rn}$ \ $(j=1,\dots,s)$.
\end{itemize}
Each $\alpha_i$ is given explicitly in terms of\/ $W$ by \eqref{alpha-L-S-U} as in Example \ref{ex:alternating}.

Further, if\/ $\Phi:M\to\Omega_r\U{n}^{\nu,\rn}$, then
each partial product $\eta_0\dots\eta_i$ has values in $\Omega_i\U n^{\nu}$.
\qed \end{theorem}

{}From the proof of Proposition \ref{pr:expl-S-U}, we see that the first uniton
$\alpha_1$ is given by $\alpha_1 =  P_0(\lambda^{-(r-1)/2} W \cap \HH_+) = \delta_{(r+1)/2}$ where
$\sum_{i=0}^{r-1}\lambda^i \delta_{i+1}+\lambda^r\HH_+$ is the $S^1$-invariant limit of $W$
(see Example \ref{ex:S1-limit}).

To apply this to harmonic maps, note that if $\Phi:M\to \Omega_r\U n^{\rn}$ is an extended solution with $r$ odd, then $\ii\Phi_{-1}$ has values in $\O n$. If, additionally, $\Phi:M\to\Omega_r\U n^{\nu,\rn}$, then $\Phi_{-1}$ has values in $\O{2m}/\U m\subset\U{2m}$. We give a converse.

\begin{lemma}
Let $\varphi:M\to\O{2m}/\U m\subset\U{2m}$ be a harmonic map of finite uniton number. Then there is an extended solution $\Phi:M\to\Omega_r\U{2m}^{\nu, \rn}$ with $r$ odd and $\Phi_{-1}=\pm\varphi$.
\end{lemma}

\begin{proof}
Let $z_0\in M$ and write $\varphi(z_0)=Y$ so that 
$Y$ is a maximal isotropic subspace of $\cn^{2m}$.
Then $Q_{\lambda}=\pi_Y+\lambda\pi_{\ov{Y}}\in\Omega\U{2m}^{\nu}$ satisfies $Q_{-1}=\varphi(z_0)$.
By Remark \ref{rem:fi-uniton}, there is an extended solution $\Psi:M\to\Omega_{\alg}\U n$ associated to $\varphi$ with $\Psi(z_0)=Q$. 
From $\ov Q_{\lambda}=\lambda^{-1}Q_{\lambda}$ and $Q_{\lambda}\in\Omega\U{2m}^{\nu}$, by uniqueness we see that $\Psi:M\to\Omega\U{2m}^{\nu}$. 
Thus, if $\Psi=\sum_{\ell=-s}^t\lambda^{\ell} T_{\ell}$ with $T_{-s}, T_t\neq 0$, then from $\ov\Psi=\lambda^{-1}\Psi$ we see that $s=t-1$. Hence $\Phi=\lambda^s\Psi:M\to\Omega_{2s-1}\U{2m}^{\nu,\rn}$ is an extended solution with $\Phi_{-1}=\pm\varphi$.
\end{proof}

Putting together the results of the last three subsections we obtain the following.

\begin{corollary} \label{cor:fact-real}
Let $W$ be an extended solution which is real of some degree $r$.  Define subbundles $\alpha_i$ by \eqref{alpha-L-S-U}
as in Example \ref{ex:alternating}, and $\varphi:M \to \O n$ by \eqref{phi-fact}.   Then $\varphi$ is a harmonic map of uniton number at most $r$ and every such harmonic map is given this way.  If\/ $W$ is closed under $\nu$, then $\varphi$ has image in
a real Grassmannian ($r$ even) or $\O{2m}/\U{m}$ ($r$ odd), and every harmonic map of finite uniton number from a surface to these spaces is obtained in this way.
\end{corollary}

We discuss how to find real extended solutions $W$ in \S \ref{subsec:real-ex}. 

\subsection{Real $S^1$-invariant solutions}
\label{subsec:s1invariant}

On $\H_+/\lambda^r \H_+
=\cn^n+\lambda\cn^n+\dots+\lambda^{r-1}\cn^n\cong\cn^{nr}$, define the bilinear form 
$$
\ip{v}{w}_s=\sum_{k=0}^{r-1}\ip{v_k}{w_{r-k-1}}_{\cn},
$$
where $v=v_0+\lambda v_1+\dots+\lambda^{r-1}v_{r-1}$, $w=w_0+\lambda w_1+\dots+\lambda^{r-1}w_{r-1}$, and 
$\ip{\cdot}{\cdot}_{\cn}$ denotes the standard complex symmetric bilinear form on $\cn^n$.  Note that $\ip{v}{w}_s$ gives the
$L^2$-inner product of $\lambda^{1-r}v$ and $\ov{w}$.
For any subspace $V\subset \H_+/\lambda^r\H_+$\,, write  
$
V^{\perp_s}=\{v\in Z\ |\ \ip{v}{w}_s=0 \text{ for all } w\in V\},
$
then the following is immediate.
 
\begin{lemma} \label{le:real}
Let $W=V+\lambda^r\H_+\in\Gr_r$, where $V\subset \H_+/\lambda^r \H_+$\,. Then $W\in\Gr_r^{\rn}$ if and only if\/ $V^{\perp_s}=V$.
In this case, $2\dim V=nr$ so that $nr$ is even.
\qed \end{lemma}

Note that the last statement follows from $\dim V^{\perp_s}=nr-\dim V$.
 
Now let $\Phi:M\to\Omega_r\U n$ be an extended solution and set $W=\Phi\H_+:M\to\Gr_r$\,. {}From Propositions \ref{pr:alg-iso} and \ref{pr:S1-invt}
and Lemma \ref{le:real}, $\Phi$ is $S^1$-invariant if and only if
$W =  V + \lambda^r \HH_+$ where
$V = \sum_{i=0}^{r-1} \lambda^i\delta_{i+1} $
for some superhorizontal sequence
(Definition \ref{def:superhor}):
\begin{equation} \label{superhor}
\ul{0} = \delta_0 \subset \delta_1 \subset \cdots \subset \delta_r \subset \delta_{r+1} = \CC^n.
\end{equation} 
As before, the $\delta_i$ are the Segal unitons for $\Phi$. 
Call this sequence \emph{real (of degree $r$)} if
$\delta_i^{\,\perp_{\cn}} = \delta_{r-i+1}$ for all $i$;  here ${}^{\perp_{\cn}}$ denotes orthogonal complement with respect to $\ip{\cdot}{\cdot}_{\cn}$.
On setting $\zeta_i = \delta_{i}^{\perp} \cap \delta_{i+1}$ \ $(i=0,\ldots, r)$ as in Proposition \ref{pr:S1-invt}, this is equivalent to $\zeta_i = \ov{\zeta_{r-i}}$\,.
Note that the sequence $(\delta_i)$ is real
if and only if $W$ is real. Furthermore, if $W:M\to\Gr_r^{\rn}$ is an arbitrary real extended solution, then its $S^1$-invariant limit (Example \ref{ex:S1-limit}) is also real. 

\begin{proposition} Let $\Phi:M\to\Omega_r\U n^{\rn}$ be an $S^1$-invariant extended solution and consider the corresponding superhorizontal sequence \eqref{superhor}.

{\rm (i)} The $\eta_i$ in Theorems \ref{th:fact-SOn} and \ref{th:fact-r-odd} all commute and
$$
\lambda^{-1}\eta_{2\ell-1}\eta_{2\ell}=\lambda^{-1}\pi_{\delta_{s-\ell+1}}+\pi_{(\delta_{s-\ell+1}+\ov{\delta_{s-\ell+1}}\,)}^{\perp}+\lambda\pi_{\ov{\delta_{s-\ell+1}}}\qquad(\ell=1,\dots,s),
$$
where $s=[r/2]$. Additionally, if $r$ is odd, $\eta_0 = \pi_V + \lambda\pi_V^{\perp}$ where $V = \delta_{(r+1)/2}$. 

{\rm (ii)} Suppose that $r=2s$ and set
$\varphi = \Phi_{-1}$.  Then $\varphi$ is given by 
$$
\varphi=\zeta_s\oplus\sum_{k=0}^{s/2-1}\zeta_{2k}\oplus\ov{\zeta_{2k}} \text{  $($$s$ even$)$; } \varphi=\sum_{k=0}^{(s-1)/2}\zeta_{2k}\oplus\ov{\zeta_{2k}} \text{  $($$s$ odd$)$};
$$
the sum is the direct sum of harmonic maps into $G_{2j}(\rn^n)$ for various $j$.
\qed \end{proposition}

We now discuss how to find real superhorizontal sequences.
Let $f:M \to G_k(\cn^n)$ be a holomorphic map.  Let $f_{(i)}$ denote its $i$'th associated curve
(Definition \ref{def:assoc-curve}), and $G^{(i)}(f)$ its $i$'th $\pa'$-Gauss transform (Example \ref{ex:Gauss-filtr});
Note that $G^{(i)}(f) = f_{(i-1)}^{\perp} \cap f_{(i)}$ (where we set $f_{(-1)}$ equal to the zero bundle).

Say that $f:M \to G_k(\cn^n)$ (or the corresponding holomorphic subbundle of $\CC^n$) has \emph{isotropy order} $t \in \{-1,0,1,\ldots\}$
if $\ip{f_{(i)}}{f_{(j)} }_{\cn}= 0$ for $0 \leq i+j \leq t$ and $\ip{f_{(i)}}{f_{(j)}}_{\cn}\neq0$ for $i+j=t+1$; if there is no such number $t$, we say that $f$ is of \emph{infinite isotropy order} or \emph{(strongly) isotropic}
\cite{erdem-wood}. For a map
$f:M \to \CP{n-1}$ we have 
$t=2s+1$ with $s \leq \bigl[(n-3)/2 \bigr]$,
see \cite{bahy-wood-G2}.  To say that  $f$ has isotropy order at least $t$ is equivalent to saying that $f_{(s)}$ is
an isotropic subbundle of $\CC^n$.  A full holomorphic map
$f:M \to \CP{n-1}$ is called \emph{totally isotropic} if it has the maximum possible finite isotropy order; it follows that $n$ is odd and $f$ has isotropy order $n-2$.  Equivalently, a full holomorphic map is totally isotropic if
its $(n-1)$st $\pa'$-Gauss transform $G^{(n-1)}(f)$
is equal to $\ov{f}$.   When $M = S^2$, all holomorphic maps of finite isotropy order are given by an algorithm in \cite{bahy-wood-G2}, see (iv) below.

\begin{example} \label{ex:sharp}
(i) When $n$ is odd, and $f:M^2 \to \CP {n-1}$ is full and totally isotropic, then $\delta_i = f_{(i-1)}$ \ $(i=1,\ldots, n-1)$ defines a real superhorizontal sequence of degree $n-1$ with
$\zeta_i = G^{(i)}(f)$.
A nice example is the Veronese map \cite{bolton-jensen-rigoli-woodward}, where all the $\zeta_i$ define minimal surfaces of constant Gauss curvature in $\CP {n-1}$ with $\zeta_{(n-1)/2}$ in $\RP {n-1}$.

(ii) When $n=2m$, let $Y$ be a constant maximal isotropic subspace of $\cn^n$ and let
$\ul 0=\delta_0\subset\delta_1\subset\dots\subset\delta_{s-1}\subset\delta_s=\ul Y$ be any superhorizontal sequence in $\ul Y$. 
For $i=1,\dots s$, set $\delta_{s+i}=\delta_{s-i}^{\perp_{\cn}}$\,. Then $(\delta_i)_{0\leq i\leq 2s}$, is a real superhorizontal sequence in $\CC^n$. 

(iii) Again with $n=2m$, to obtain an example with no $\delta_i$ constant,  choose $h:M\to\cn P^{n-1}$ full and totally isotropic, then set $\delta_i = h_{(i-1)}$ \ $(i=1,\ldots, m-1)$ and $\delta_i = \delta_{n-i-1}^{\perp_{\cn}}$ \
$(i = m,\ldots n-2)$; this is real of degree $n-2$ and normalized.

(iv)  Take isotropic coordinates on $\cn^4$ and $\cn^6$ and
define $h_0:S^2 = \CP 1 = \cn \cup \{\infty\} \to \CP 3$ by $h_0(z)=[1,z^3,z,-z^2]$.  Then $h_0$ is full holomorphic and isotropic.
Applying the algorithm of \cite{bahy-wood-G2} yields a full totally isotropic holomorphic map
$h:S^2 \to \CP 5$  given by $h(z) = [24 z, 6z^4, 12z^2, -8z^3,1, -48z^5]$. By part (iii), $\ul 0\subset h\subset h_{(1)}\subset h_{(1)}^{\perp_{\cn}}\subset h^{\perp_{\cn}}\subset\CC^6$ is a real superhorizontal sequence. 
\end{example}

The following result will be useful when obtaining bounds on the uniton numbers of real extended solutions. 

\begin{lemma}\label{le:isotropic} Let
$X\subset\CC^{2m}$ be a maximal isotropic holomorphic subbundle with\\ $\dim X_{(1)}\leq\dim X+1$. Then $X$ is constant.
\end{lemma}

\begin{proof}  Define a complex-valued $2$-form on $X$ by 
$$
\rho_p(v,w)=\ip{\pa\sigma(p)}{w}_{\cn}
\qquad(v,w\in X_p,\ p \in M),
$$
where $\sigma$ is any local section of $X$ with
$\sigma(p)=v$. For a fixed $p \in M$, consider 
$$
F:X_p \to X_p^*,\quad F(v)=\rho_p(v,\cdot).
$$
Clearly, $\rho$ restricts to a non-degenerate $2$-form on $X_p/\ker F$; in particular, $X/\ker F$ must be of even dimension. However, since $X$ is maximal isotropic, we have the exact sequence 
$$
0\to X_p \hookrightarrow (X_{(1)})_p \to \Ima F\to 0,
$$
so that $\dim X_p/\ker F=\dim\Ima F\leq 1$. Hence
$\Ima F=0$, so $X_p=(X_{(1)})_p$\,, i.e., $X$ is constant. 
\end{proof}

It is easy to find all real superhorizontal sequences of
\emph{odd} degree as follows.

\begin{proposition} All real superhorizontal sequences $(\delta_i)$ of odd degree $r$ are given inductively in the following way$:$
\begin{enumerate}
\item[(i)] choose a maximal isotropic holomorphic subbundle $Y$ of\/ $\CC^{2m}$, i.e., a holomorphic map $Y \to \U{m}/\O{2m}$  (which may be constant), and set
$\delta_{(r+1)/2} = Y$.

\item[(ii)] given $\delta_{i+1}, \cdots, \delta_{(r+1)/2}, \cdots \delta_{r-i}$ for some $i \in \{2,\ldots, (r-1)/2\}$, choose a holomorphic subbundle $\delta_i$ of $\CC^n$ satisfying 
{\rm (a)}  $\delta_i \subset \{(\delta_{r-i})_{(1)}\}^{\perp_{\cn}}$
and set\\
\indent {\rm (b)}  $\delta_{r-i+1} = \delta_i^{\,\perp_{\cn}}$.
\qed
\end{enumerate}
\end{proposition}

Note that, apart from the initial choice of $\delta_{(r+1)/2}$, we only have to differentiate and solve the linear equations (a) and (b) to find
the other $\delta_i$. Note, further, that if $\{(\delta_{r-i})_{(1)}\}^{\perp_{\cn}}$ is the zero bundle, $\delta_j$ for $j \leq i$ will be zero and we will effectively have a shorter sequence.  

An alternative way of finding real superhorizontal sequences is illustrated by the following example; it can readily be generalized to find real superhorizontal sequences of any (even or odd) degree with subbundles of specified ranks. 

\begin{example} \label{ex:super135}
To find all superhorizontal sequences of degree $3$ in $\CC^6$
with the $\delta_i$ of ranks $1$, $3$ and $5$, let
$h_0:M \to \cn P^5$ be a holomorphic map of isotropy order at least $2$ and set $\delta_1 = h_0$.
If $h_0$ is constant, choose $h_3$ to be a rank two isotropic
 subbundle of the constant subbundle $(h_0)^{\perp_{\cn}}$ not containing $h_0$.
 If $h_0$ is non-constant,
choose $h_3$ to be a rank one isotropic subbundle of ${(h_0)_{(1)}}^{\!\!\perp_{\cn}}$ not lying in
$(h_0)_{(1)}$; possible subbundles can be found explicitly by solving linear equations to find a basis for ${(h_0)_{(1)}}^{\!\!\perp_{\cn}}$, then the coefficients of a meromorphic section $H_3$ spanning $h_3$ will satisfy a single quadratic equation.  In either case, set
$\delta_2 = (h_0)_{(1)} + h_3$, and
$\delta_3 = (\delta_1)^{\perp_{\cn}}$.
\end{example} 

\subsection{Normalization of real solutions}\label{subsec:normreal}

Given an extended solution $\Phi:M\to\Omega_r\U n^{\rn}$, as usual set $W=\Phi\H_+:M\to\Gr_r^{\rn}$ and $\delta_{i+1}=P_i(W\cap\lambda^i\HH_+)$ \
$(0\leq i\leq r-1)$. From the reality conditions, it follows that $\delta_{(r+1)/2}$ is maximal isotropic when $r$ is odd. In fact,
as remarked after Theorem \ref{th:fact-r-odd}, this is the first uniton in the alternating uniton factorization of $W$. 
We may perform a normalization procedure similar to that of the $\U n$ case discussed in \S \ref{subsec:normalized}.

\begin{proposition} \label{pr:max-degree} 
Let $\Phi:M \to \Omega_r\U n^{\rn}$ be an extended solution with $r\geq3$.  Then $nr$ is even and
\begin{enumerate}
\item[(i)] there is a loop $\eta \in\Omega_{\alg}\U n^{\rn}$
 such that $\wt{\Phi} = \eta^{-1}\Phi$ is real and
normalized\/$;$

\item[(ii)] the degree of any normalized extended solution $\wt{\Phi}$ is at most $n-1 ;$ 

\item[(iii)] If\/ $n$ is even, we find $\eta$ such that the degree of $\wt{\Phi}$ at most $n-2$.
\end{enumerate}
\end{proposition}

\begin{proof} Set $W=\Phi\H_+:M\to\Gr_r^{\rn}$.  That $nr$ is even follows from Lemma \ref{le:real}.

(i) Suppose that $A_i=0$ for some $i$. Then $A_{r-i}=0$ as well, so we may assume that $i\geq r/2$. If $i=r$, then $\lambda^{-1}W:M\to\Gr_{r-2}^{\rn}$. If $r$ is odd and $i=(r+1)/2$, then $\delta_{(r+1)/2}=\delta_{(r+3)/2}$, so $Y=\delta_{(r+1)/2}$ is constant and maximal isotropic. As in Lemma \ref{le:reduce}, it follows that $(\pi_Y+\lambda^{-1}\pi_Y^{\perp})W:M\to\Gr_{r-1}^{\rn}$. Similarly, if $r$ is even and $i=r/2$, then $Y=\delta_{r/2}=\delta_{r/2+1}$ 
is constant and maximal isotropic, and $(\pi_Y+\lambda^{-1}\pi_Y^{\perp})W:M\to\Gr_{r-1}^{\rn}$. 
It remains to consider the case when $i>(r+1)/2$. Following \cite{pacheco-sympl}, we set 
$$
V=\lambda^{i+1-r}(W\cap\lambda^{r-i-1}\HH_+)+\lambda^{2-i}(W\cap\lambda^{i-1}\HH_+)+\lambda^2\HH_+\,.
$$
We have $\lambda^{r-2}V\subset W\subset V$; from the fact that $A_i=A_{r-i}=0$, it follows easily that $V$ is real of degree $2$ and constant. Hence $V=\eta\HH_+$ for some constant $\eta\in\Omega_2\U n^{\rn}$ and $\eta^{-1}W:M\to\Gr_{r-2}^{\rn}$.

(ii) Since $A_i\neq 0$ for all $i$, we have $n=\sum_{i=0}^r\rank A_i\geq r+1$, and the inequality follows. 

(iii) When $n$ is even and $r=n-1$, either some $A_i=0$ or $\dim\delta_i=i+1$ for all $i$.  If $A_i=0$ for some $i$, then we can reduce the degree by (i).  If $\dim\delta_i=i$ for all $i$, then $Y=\delta_{(r+1)/2}$ is constant by Lemma \ref{le:isotropic}. It is easy to see that $(\pi_Y+\lambda^{-1}\pi_Y^{\perp})W:M\to\Gr_{r-2}^{\rn}$. 
\end{proof}

\begin{remark} The bounds are sharp as shown by Examples \ref{ex:sharp} (i) and (iii).  
\end{remark}

\subsection{Examples of real harmonic maps and explicit formulae} \label{subsec:real-ex}

To find (extended solutions of) harmonic maps into the orthogonal group, real Grassmannians and the space
$\O{2m}/\U{m}$ we need to find extended solutions $W$ which are real (Definition 6.1).
To do this, we start with a finite set of meromorphic sections $H_j$ of $\CC^n$, set $X$ equal to their span
and compute the corresponding extended solution $W:M \to \Gr$ from \eqref{W}, then we impose the
reality conditions: these are linear and quadratic equations.  As in Example \ref{ex:alternating}, the  unitons $\alpha_i$  for the alternating factorization
are then given by  \eqref{alpha-L-S-U}; the harmonic map and associated extended solution are then the products \eqref{phi-fact}, \eqref{Phi-fact} of those unitons.

Let $\Phi:M\to\Omega_r\U n^{\rn}$ be an extended solution and set $W=\Phi\H_+:M\to\Gr_r^{\rn}$. As above,
we shall assume that $\Phi$ is normalized and that,
if $r$ is odd, $\delta_{(r+1)/2}$ is non-constant. Then $r$ is even and $\leq n-1$ if $n$ is odd and $r \leq n-2$ if $n$ is even.   Write $\varphi = \Phi_{-1}:M \to \U n$, so that $\ov\varphi=\pm\varphi$.

In the following results and examples, for an extended solution $W$ we write, as usual,
$\delta_{i+1}=P_i(W\cap\lambda^i\HH_+)$ so that the
$S^1$-invariant limit (Example \ref{ex:S1-limit}) of $W$ is
$W^0 = \sum_{i=0}^{r-1}\lambda^i \delta_{i+1}+\lambda^r\HH_+$.
As before, for a meromorphic section $H$ of $\CC^n$, we write $H^{(i)}$ to mean its $i$'th derivative with respect to a local complex coordinate on $M$;
further, we write $(H)_{(i)}$ to mean the osculating subbundle
 $h_{(i)}$ where $h = \spa\{H\}$; thus
$(H)_{(i)} = \spa\{H,H^{(1)}, \ldots, H^{(i)}\}$.

We first examine some low values of $r$ and $n$.  Clearly $r=0$ if and only if $\Phi=I$; the case $r=1$ is covered by Example \ref{ex:r=1real}. 

\begin{proposition}\label{pr:r2} Suppose that $r=2$ and $\rank\delta_1=1$. Then $W$ is $S^1$-invariant and $\varphi=\delta_1+\ov{\delta}_1:M\to G_2(\rn^n)$ so that $\varphi$ is a \emph{real mixed pair} in the sense of\/
\cite{bahy-wood-G2}. 
\end{proposition}

\begin{proof}  We have $W=\spa\{H_0+\lambda H_1\}+\lambda\delta_2+\lambda^2\HH_+\,.$
However, the reality condition on $W$ implies that $H_1\in\delta_1^{\perp_{\cn}}=\delta_2$ and hence $W=\delta_1+\lambda\delta_2+\lambda^2\HH_+\,,$ so that
$\varphi$ is a mixed pair as in
Example \ref{ex:superhor}(i). 
\end{proof}

\begin{proposition} \label{pr:n4}
If $n \leq 4$ then $\Phi$ is $S^1$-invariant.
\end{proposition}

\begin{proof}  If $n=2$, then $r=0$ so that $\varphi$ is constant.

If $n=3$, $\rank\delta_1=1$ and the result follows from Proposition  \ref{pr:r2}. 

If $n = 4$, by Proposition \ref{pr:max-degree}, $r=0$, $1$ or $2$. If $r=2$, then $\rank\delta_1=1$ and the result follows from Proposition \ref{pr:r2}. 
\end{proof}

The last two propositions are sharp as is shown by
Examples \ref{ex:n5r2rank2} and \ref{ex:r3n6}.

For general $n$ and $r$, all $W$ can be found by solving linear and quadratic equations; we illustrate this with some examples.  

\begin{example}\label{ex:n5r2rank2} Let $n=5$ and $r=2$ and set
$$
W = \spa\{H_0+\lambda H_1\,,\, H_0^{(1)} + \lambda H_3\}
+\lambda (H_0)_{(2)} + \lambda^2\HH_+
$$
where $H_0$ is full. It is easily checked that $W:M\to\Gr_2$ is an extended solution. Then $W$ is real if and only if 
(i) $H_0$ (i.e., $\spa\{H_0\}$) is totally isotropic;
(ii) $\ip{H_1}{H_0}_{\cn} = 0$;
(iii) $\ip{H_3}{H_0^{(1)}}_{\cn} = 0$;
(iv) $\ip{H_3}{H_0}_{\cn} + \ip{H_1}{H_0^{(1)}}_{\cn}= 0$.
The uniton factorization of Theorem \ref{th:fact-SOn} is $\Phi=(\pi_{\alpha_1}+\lambda\pi_{\alpha_1}^{\perp})(\pi_{\alpha_2}+\lambda\pi_{\alpha_2}^{\perp})$, where $$
\alpha_1=(H_0)_{(2)} \quad \text{and} \quad
\alpha_2=\spa\{H_0+\pi_{\alpha_1}^{\perp} H_1, H_0^{(1)}+\pi_{\alpha_1}^{\perp} H_3\}. 
$$
As before, when $M = S^2$, all solutions of (i) are given by the algorithm in
\cite{bahy-wood-G2}.  Set $H_1 = a H_0^{(3)}$ and
$H_3 = b H_0^{(3)} + c H_0^{(4)}$ for some meromorphic functions $a$, $b$ and $c$.  Then (ii) is satisfied and (iv) reads
$c\ip{H_0}{H_0^{(4)}}_{\cn} + a\ip{H_0^{(3)}}{H_0^{(1)}}_{\cn} = 0$, 
i.e.,
$(c-a)\ip{H_0}{H_0^{(4)}}_{\cn} = 0$ which is satisfied if $c=a$.  Finally (iii) can be written:
$b\ip{H_0}{H_0^{(4)}}_{\cn} = c\ip{H_0^{(1)}}{H_0^{(4)}}_{\cn}$.
Since $H_0$ is full, $\ip{H_0}{H_0^{(4)}}_{\cn}$ is not identically zero, so we can find $b$ to satisfy this.

This determines an extended solution $W=\Phi\H_+$ of a harmonic map $\varphi = \Phi_{-1}:M \to \O{5}$.  If $a$ is identically zero, then $\Phi$ is $S^1$-invariant with
$\varphi = H_{(1)} \oplus \ov{H_{(1)}}$ so that
$\varphi^{\perp} = G^{(2)}(H_0)$ is a harmonic
 map into $\RP{4}$.  If $a$ is not identically zero, then
$\Phi$ does not have values in $\Omega_2\U 5^{\nu, \rn}$ and
$\varphi$ does not lie in a Grassmannian.
\end{example}

\begin{example}\label{ex:r3n6}
Let $n=6$ and $r=3$ and set $X=\spa\{H_0+\lambda H_1+\lambda^2H_2, \lambda H_3+\lambda^2H_4\}$. Applying \eqref{W} gives the extended solution 
\begin{equation*}
\begin{split}
W &=\spa\{H_0+\lambda H_1+\lambda^2 H_2\} + \lambda \,\spa\{H_3+\lambda H_4, (H_0+\lambda H_1)_{(1)}\}\\
& + \lambda^2\spa\{(H_0)_{(2)}, (H_3)_{(1)}\}+\lambda^3\HH_+\,.
\end{split}
\end{equation*}
The $S^1$-invariant limit is
$W^0 =  \delta_1 + \lambda\delta_2+\lambda^2\delta_3+\lambda^3\HH_+$ where
$\delta_1= \spa\{H_0\}$,
$\delta_2 = \spa\{(H_0)_{(1)},H_3\}$, and
$\delta_3 = (H_0)_{(2)}+(H_3)_{(1)}$. This is a
real superminimal sequence if and only if $\delta_2$ is a
maximal isotropic subbundle of $\CC^6$ and $\delta_3=\delta_1^{\perp_{\cn}}$.  We can find all possible $H_0$ and $H_3$ as in Example \ref{ex:super135}.

Next, $W$ is real if and only if, additionally,
$$
\left. \begin{array}{rlrl}
{\rm (i)} & \ip{H_0}{H_2}_{\cn}+\ip{H_1}{H_1}_{\cn} &=0, \\
{\rm (ii)}& \ip{H_0}{H_1}_{\cn} &= 0, \\
{\rm (iii)} & \ip{H_1}{H_3}_{\cn}+\ip{H_0}{H_4}_{\cn} &=0. 
\end{array} \right\}
$$
We  can find all possible $H_1$ by solving the linear equation (ii); this is equivalent
to choosing a meromorphic section of 
$(\delta_1)^{\perp_{\cn}} = \delta_3$.    We then find all possible $H_2$ by solving the linear equation (i). 
Finally we solve the linear equation (iii) to find all possible $H_4$.  Then $W=\Phi\H_+$ is an extended solution with $\ii\Phi_{-1}:M\to\O 6$. We can easily calculate the factorization of Theorem \ref{th:fact-SOn} as in the previous example. 
\end{example}

\subsection{Maps into real Grassmannians and $\O{2m}/\U{m}$} 

For real $\nu$-invariant extended solutions, we obtain the following bounds on the minimal uniton number. 
\begin{proposition} Let $\Phi:M\to\Omega_r\U n^{\nu,\rn}$ be an extended solution with $r\geq 4$. Then 
\begin{enumerate}
\item[(i)] there exists $\eta\in\Omega_{r-s}\U n^{\rn}$ such that $\wt\Phi=\eta^{-1}\Phi:M\to\Omega_s\U n^{\nu,\rn}$ is normalized$;$

\item[(ii)] if\/ $\Phi_{-1}:M\to G_k(\rn^n)$, then $s\leq2\min\{k-1,n-k\} ;$

\item[(iii)] if\/ $\Phi_{-1}:M\to\O{2m}/\U m$, then $\eta$ may be chosen so that $s\leq 2m-3$.
\end{enumerate}
\end{proposition}

\begin{proof} (i) Suppose that $\Phi$ is not normalized, so that $A_i=0$ for some $i$. Since this implies that $A_{r-i}=0$, we may assume that $i\geq r/2$. The result is obvious if $i=r$. If $r$ is odd and $i=(r+1)/2$, then $Y=\delta_{(r+1)/2}=\delta_{(r+3)/2}$ is constant and maximal isotropic. As in Proposition \ref{pr:reduce-Gr} it follows that $(\pi_Y+\lambda^{-2}\pi_Y^{\perp})$, $W:M\to\Gr_{r-2}^{\nu,\rn}$. Similarly, if $r$ is even and $i=r/2$, then $Y=\delta_{r/2}=\delta_{r/2+1}$ is constant and maximal isotropic, and $(\pi_Y+\lambda^{-2}\pi_Y^{\perp})W:M\to\Gr_{r-2}^{\nu,\rn}$\,. 

Suppose now that $r/2+1\leq i\leq r-1$. Following the treatment of the symplectic case in
\cite{pacheco-sympl} we define 
$$
V=\lambda^{i+1-r}(W\cap\lambda^{r-i-1}\HH_+)+\lambda^{3-i}(W\cap\lambda^{r-i-1}\HH_+)+\lambda^4\HH_+\,.
$$
We have $\lambda^{r-4}V\subset W\subset V$; from the fact that $A_i=A_{r-i}=0$, it follows easily that $V$ is constant and lies in $\Gr_4^{\rn}$. As in \cite{pacheco-sympl}, we see that 
$
(V\ominus\lambda V)^{\text{odd}}=0.
$
Hence $V=\eta\HH_+$\,, for some $\eta\in\Omega_4\U n^{\rn}$ with $\eta(-1)=I$, and $\eta^{-1}W:M\to\Gr_{r-4}^{\nu,\rn}$\,.

(ii) We know that $r$ is even, and since $\wt W=\wt \Phi\H_+$ is normalized, we have 
$$
k=\sum_{j=0}^{s/2}\rank A_{2j}\geq s/2+1\text{ and } n-k=\sum_{j=0}^{s/2-1}\rank A_{2j+1}\geq s/2,
$$
and the inequality follows. 

(iii) Since $r$ is odd, $\delta_{(r+1)/2}$ is maximal isotropic. If $Y=\delta_{(r+1)/2}$ is constant, then $(\pi_Y+\lambda^{-2}\pi_Y^{\perp})\wt\Phi:M\to\Omega_{s-2}\U n^{\nu,\rn}$. Thus, assume that $\delta_{(r+1)/2}$ is non-constant. By Proposition \ref{pr:max-degree} we have $\rank A_{(r+1)/2}=\rank\delta_{(r+3)/2}-\rank\delta_{(r+1)/2}\geq 2$. By reality, we then also have $\rank A_{(r-1)/2}\geq 2$. Thus 
$$
2m=\sum_{j=0}^s\rank A_j\geq s+3,
$$
and the inequality follows. 
\end{proof}

\begin{corollary} {\rm (i)} Suppose that
$k$ or $n-k$ is even and let $\varphi:M\to G_k(\rn^n)\subset\U n$ be a harmonic map of finite uniton number. Then  there is an extended solution $\Phi:M\to\Omega_r\U n^{\nu,\rn}$ with $r\leq2\min\{k,n-k\}$ and $\Phi_{-1}=\pm\varphi$. 

{\rm (ii)} Let $\varphi:M\to\O{2m}/\U m\subset\U{2m}$ be a harmonic map of finite uniton number. Then there is an extended solution $\Phi:M\to\Omega_r\U{2m}^{\nu,\rn}$ with $r\leq 2m-3$ and $\Phi_{-1}=\pm\varphi$. 
\qed \end{corollary}
Note that if $k$ and $n-k$ are odd, then we can embed $G_k(\rn^n)$ in $G_k(\rn^{n+1})$. 

We now give some examples, starting with a classification for low values of $r$.

\begin{proposition}

Let $W:M\to\Gr_r^{\nu,\rn}$ be an extended solution.  Suppose that $r \leq 2$, or $r=3$ and
$\rank\delta_1=1$. Then $W$ is $S^1$-invariant.
\end{proposition}

\begin{proof}
When $r=2$, this follows from Example \ref{ex:superhor}.

When $r=3$,
 $
W=\spa\{H_0+\lambda^2H_1\}+\lambda\delta_2+\lambda^2\delta_3+\lambda^3\HH_+\,.
$
However, the reality conditions imply that $H_1\in\delta_1^{\perp_{\cn}}=\delta_3$, and hence
$W = \delta_1 +\lambda\delta_2+\lambda^2\delta_3+\lambda^3\HH_+$,
which is $S^1$-invariant.
\end{proof}

The next two examples show that the hypotheses
of the last result are sharp.

\begin{example} \label{ex:m4+r3}
Let $m \geq 4$ and $r=3$.  Consider the extended solution 
$$
W = \spa\{H_0 + \lambda^2 H_1, H_2+ \lambda^2 H_3\}
+ \lambda \delta_2 + \lambda^2 \delta_3 + \lambda^3\HH_+\,,
$$ 
where $\ul 0\subset\delta_1=\spa\{H_0,H_2\}\subset\delta_2\subset\delta_3\subset\CC^{2m}$ is a superhorizontal sequence. 
When $H_1$ and $H_3$ are zero, we obtain the $S^1$-invariant limit
$W^0 = \delta_1 + \lambda \delta_2 + \lambda^2 \delta_3 + \lambda^3\HH_+\,.
$
Then $W$ is real if and only if the sequence $(\delta_i)$ is real and 
\begin{equation}\label{2m-3-reality}
\ip{H_0}{H_1}_{\cn}=\ip{H_2}{H_3}_{\cn}=\ip{H_0}{H_3}_{\cn}+\ip{H_1}{H_2}_{\cn}=0.
\end{equation}

The filtration of $W$ given by alternating Uhlenbeck and Segal steps is given by
\begin{align*}
W_1 &= \lambda^{-1}W \cap \HH_+ + \lambda\HH_+
	= \delta_2 + \lambda\HH_+,\\
W_2 &= \lambda^{-1}W \cap \HH_+ + \lambda^2\HH_+
= \delta_2 +\lambda \delta_3 + \lambda^2\HH_+;
\end{align*}
the resulting factorization of Theorem \ref{th:fact-r-odd} is
$\Phi = \Pi_{j=1}^i(\pi_{\alpha_j} + \lambda \pi_{\alpha_j}^{\perp})$ where 
\begin{align*}
\alpha_1 &= P_0 W_1 = \delta_2\,,\\
\alpha_2 &= \sum_{s=0}^1 S^1_s P_s W_2
= \pi_{\alpha_1}\delta_2 \oplus \pi_{\alpha_1}^{\perp}\delta_3 = \delta_3\,,\\
\alpha_3 &= \sum_{s=0}^2 S^2_s P_s W_2
= \spa\{H_0+\pi_{\delta_3}^{\perp}H_1, H_2+\pi_{\delta_3}^{\perp}H_3\}\,.
\end{align*}
The corresponding harmonic maps $\varphi_i  =
\Pi_{j=1}^i(\pi_{\alpha_j} - \pi_{\alpha_j}^{\perp})$ are:
\begin{align*}
\varphi_1 &= \delta_2:M^2 \to \O {2m}/\U m\,,\\
\varphi_2 &= \zeta_2^{\perp}:M^2 \to G_*(\cn) \quad
\text{where} \quad \zeta_2 = \delta_2^{\perp} \cap \delta_3\,,\\
\varphi &= \varphi_3
= \alpha_3 \oplus \zeta_2:
M^2 \to \O {2m}/\U m\,. 
\end{align*}
The harmonic map $\varphi^0$ corresponding to the $S^1$-invariant limit of $W$ is given by
$\varphi^0 = \delta_1 \oplus \zeta_2:M^2 \to \O{2m}/\U{m}$.

As an example, we find solutions $W$ with $m=5$.
Consider a natural decomposition $\cn^{10}=\cn^5_1\oplus\cn^5_2$ of $\cn^{10}$ into two copies of $\cn^5$; choose full totally isotropic meromorphic maps $H_0:M\to\cn^5_1$ and $H_2:M\to\cn^5_2$.
Set $\delta_1=\spa\{H_0,H_2\}$. Next, pick $H_4:M\to((\delta_1)_{(1)})^{\perp_{\cn}} \setminus (\delta_1)_{(1)}$ meromorphic with $\ip{H_4}{H_4}_{\cn}=0$. Setting $\delta_2=(\delta_1)_{(1)}$ and $\delta_3=\delta_1^{\perp_{\cn}}$ defines a real superminimal sequence giving an $S^1$-invariant harmonic map $M \to \O {10}/\U 5$.

Next, choose $H_1 = K_1 + a H_2^{(4)}$ and
$H_3 = b H_0^{(4)} + K_2$.  Then the reality conditions
\eqref{2m-3-reality} are satisfied if
$K_1:M\to\spa\{H_0\}^{\perp_{\cn}} \cap \cn^5_1$, 
$K_2:M\to\spa\{H_2\}^{\perp_{\cn}} \cap \cn^5_2$,
and
$$
b\ip{H_0}{H_0^{(4)}}_{\cn} + a\ip{H_2}{H_2^{(4)}}_{\cn} = 0.
$$
Neither $\ip{H_0}{H_0^{(4)}}_{\cn}$ nor $\ip{H_2}{H_2^{(4)}}_{\cn}$ are identically zero, so the last equation has solutions with neither $a$ nor $b$ identically zero and we have a real solution $W$.

Note that $W$ is $S^1$-invariant if and only if $H_1$ and $H_3$ are sections of $\delta_3$\,; this holds precisely when $a=b=0$. 
\end{example}  

\subsection{Harmonic maps to the symplectic group and its quotients}

Let $J$ denote the conjugate linear map given by left multiplication by the unit quaternion $j$ on $\cn^{2m}\cong\hn^m$. In a similar way to Definition \ref{def:real}, we say that $W\in\Gr_r$ is \emph{symplectic (of degree $r$)} if $J W=\lambda^{1-r}W$. We denote by $\Gr_r^J$ the set of elements of $\Gr_r$ which are symplectic of degree $r$; similarly, we denote by $\Omega_r\U{2m}^J$ the elements of $\Omega_r\U{2m}$ satisfying $J\Phi J^{-1}=\lambda^{-r}\Phi$; when $W=\Phi\H_+$\,, we have
$W\in\Gr_r^J$ if and only if $\Phi\in\Omega_r\U{2m}^J$. If $\Phi\in\Omega_r\U{2m}^J$, then $\Phi_{-1}$ or $i\Phi_{-1}$ is in $\Sp m$ depending on whether $r$ is even or odd. The results described in \S\S\ref{subsec:SOn} --- \ref{subsec:ocs} have obvious analogues when $\O n$ is replaced by $\Sp m$.
 
For the first two subsections, where $r$ is even, this was done in \cite{pacheco-sympl}, obtaining harmonic maps into $\Sp m$ and quaternionic Grassmannians, see also \cite{dong-shen} and \cite{he-shen2}.  Regarding the new results in \S \ref{subsec:ocs} for $r$ odd, the first term in the factorization described in Theorem \ref{th:fact-r-odd} will correspond to a holomorphic map into the Hermitian symmetric space $\Sp m/\U m$.  In all cases, our methods give new explicit formulae for the harmonic maps and their extended solutions.

As in \S\ref{subsec:normreal}, we can normalize any extended solution $\Phi:M\to\Omega_r\U{2m}^J$, and easily obtain the following bounds on the uniton number $r$.
\begin{proposition}\label{pr:normalsymp}
{\rm (i)} Given a harmonic map $\varphi:M\to\Sp m$ of finite uniton number, there is an extended solution $\Phi:M\to\Omega_r\U{2m}^J$ with $\Phi_{-1}=\pm\varphi$ or $\Phi_{-1}=\pm\ii\varphi$ and $r\leq 2m-1 ;$

{\rm (ii)} if $\varphi$ takes values in $\Sp m/\Sp k\times\Sp{m-k}$, then $\Phi$ can be chosen with $r\leq 4\min\{k,m-k\};$
 if $2k=m$, then this can be improved to $r\leq 4k-2$.
\qed \end{proposition}

As in the $\O n$ case, it follows easily that, if $\varphi^2=I$, then this extended solution $\Phi$ can be chosen to take values in $\Omega_r\U{2m}^{\nu,J}=\Omega_r\U{2m}^J\cap\Omega_r\U{2m}^{\nu}.$
Let $\omega$ be the standard complex skew-symmetric form on $\cn^{2m}$  preserved by $\Sp m$. A subspace $V\subset\cn^{2m}$ is said to be \emph{$J$-isotropic} if $\omega(v,w)=0$ for all $v,w\in V$. To construct examples of extended solutions into $\Omega_r\U{2m}^J$, we may proceed as in the $\O n$-case, replacing `isotropic' by `$J$-isotropic', the symmetric bilinear form $\ip{\cdot}{\cdot}_{\cn}$ by $\omega$, and complex conjugation by multiplication by $J$.

We define \emph{$J$-isotropy order} in an analogous way to the definition of isotropy order in \S\ref{subsec:s1invariant}. However, this time, for a map $f:M \to \CP{2m-1}$, the isotropy order $t$ is even and $t \leq 2m-2$.   A full holomorphic map
$f:M \to \CP{2m-1}$ is called \emph{totally $J$-isotropic} if it has the maximum possible finite $J$-isotropy order $2m-2$.  Equivalently, $f$ is totally isotropic if $G^{(2m-1)}(f) = Jf$.   When $M = S^2$, all holomorphic maps of finite $J$-isotropy order are given by an algorithm in \cite{bahy-wood-HPn}.

\begin{example} Let $h:M\to\CP{2m-1}$ is a full totally
$J$-isotropic holomorphic map. Then, as in \S\ref{subsec:s1invariant}, we can construct an $S^1$-invariant extended solution $W =\Phi\H_+ =  \sum_{i=0}^{r-1} \lambda^i\delta_{i+1}  + \lambda^r \HH_+$
from the superhorizontal sequence 
$$
\ul 0\subset\delta_1\subset\delta_2\subset\dots\subset\delta_{2m-1}\subset\CC^{2m},
$$
where $\delta_i=h_{(i-1)}$. Then $\Phi$ takes values in $\Omega_{2m-1}\U{2m}^{\nu,J}$ and
$\varphi=\Phi_{-1}:M \to \Sp m/\U m \subset \U{2m}$ is the harmonic map of (minimal) uniton number $2m-1$ given by \eqref{phi-nested} . 
\end{example}

There is no analogue of Propositions \ref{pr:r2} or of \ref{pr:n4} as the next two examples show.
Let $H_0, H_1, H_2$ be meromorphic sections of $\CC^{2m}$. 

\begin{example} Let $r=2$ and $m \geq 2$. Choose $H_0$ to have $J$-isotropy order at least $2$ and set $W = \spa\{H_0+\lambda H_1\} + \lambda \delta_2 + \lambda^2\HH_+$ where $\delta_2$ is the orthogonal complement of $\delta_1=\spa\{H_0\}$ with respect to $\omega$.
Then $W$ is a symplectic extended solution, so that $W=\Phi\H_+$ for an extended solution $\Phi:M \to \Omega_2\U{2m}^J$, thus giving a harmonic map
$\varphi=\Phi_{-1}: M \to \Sp{m}$.  We can take $H_1$ not lying in $\delta_2$, then $\varphi$ does not lie in a quaternionic Grassmannian.
\end{example}

\begin{example} Let $m=2$ and $r=3$. Set $X=\spa\{H_0+\lambda H_1\}$ and let $W$ be the extended solution given by \eqref{W}, i.e.,  
$$
W=\spa\{H_0+\lambda H_1\}+\lambda(H_0+\lambda H_1)_{(1)}+\lambda^2(H_0)_{(2)}+\lambda^3\HH_+\,.
$$
Then $W$ is symplectic if and only if
(i) $H_0$ is totally $J$-isotropic, and
(ii)  $\omega(H_0,H_1^{(1)})+\omega(H_1,H_0^{(1)})=0$.
 Following the algorithm in \cite{bahy-wood-HPn}, we can
 construct all $H_0$ satisfying (i); we may then choose $H_1$ satisfying the linear equation (ii). Provided (iii) $H_1$ does not lie in $(H_0)_{(1)}$, the extended solution
 $W$ is not $S^1$-invariant.
 
As a specific example, starting with $F_0 = (1,z)$, the algorithm of \cite{bahy-wood-HPn} yields the totally $J$-isotropic map $H_0:S^2 \to \cn^4$ of $J$-isotropy order $2$ given by
$H_0= \bigl(z, \frac{1}{2}z^2, 1, -\frac{1}{6}z^3 \bigr)$.
Set $H_1 = (0,0,0,a)$ where $a$ is a non-zero constant.  Then (ii) and (iii) are satisfied, so we obtain a harmonic map
$\varphi:S^2 \to \Sp 2\subset\U 4$ of (minimal) uniton number $3$ which does not lie in $\Sp 2/\U 2$. 
\end{example}

\begin{example} Let $m=r=3$ and choose $H_0$ full and totally $J$-isotropic. Set 
$$
W=\spa\{H_0+\lambda^2H_1\}+\lambda(H_0)_{(1)}+\lambda\spa\{H_2\}+\lambda^2(H_0)_{(2)}+\lambda^2(H_2)_{(1)}+\lambda^3\HH_+\,.
$$
Then $W$ is a $\nu$-invariant extended solution which is symplectic if and only if $H_2$ is a section of $(H_0)_{(3)}$. With these choices, $W$ represents a harmonic map into $\Sp 3/\U 3$ of (minimal) uniton number $3$. Provided $H_1$ is not a section of $(H_0)_{(2)}+(H_2)_{(1)}$, the extended solution $W$ is not $S^1$-invariant. 
\end{example}

In general, we have a similar result to Corollary \ref{cor:fact-real}; in conclusion, we have found explicit algebraic formulae which give all harmonic maps of finite uniton number from a Riemann surface to a classical compact Lie group or inner symmetric space of it.

\end{document}